\documentclass[11pt]{amsart}
\usepackage{}
\usepackage{mathrsfs}
\usepackage{amsfonts}

\usepackage{amssymb,amsthm,amsmath,hyperref,txfonts}

\usepackage{graphicx,float}

\numberwithin{equation}{section}

\newtheorem{thm}{Theorem}[section]
\newtheorem{coro}{Corollary}[section]
\newtheorem{lem}{Lemma}[section]
\newtheorem{rem}{Remark}[section]
\newtheorem{prop}{Proposition}[section]
\newtheorem{defn}{Definition}[section]

\newcommand{\beq}{\begin{eqnarray}}
\newcommand{\eeq}{\end{eqnarray}}
\newcommand{\beqno}{\begin{eqnarray*}}
\newcommand{\eeqno}{\end{eqnarray*}}
\newcommand{\be}{\begin{equation}}
\newcommand{\ee}{\end{equation}}
\newcommand{\beno}{\begin{equation*}}
\newcommand{\eeno}{\end{equation*}}

\newcommand\dl{\delta}
\newcommand\tl{\tilde}

\newcommand\Dl{\Delta}

\newcommand\N{\mathbb{N}}
\newcommand\nb{\nabla}
\newcommand\nn{\nonumber}
\newcommand{\R}{\mathbb{R}}
\newcommand\Z{\mathbb{Z}}
\newcommand\sig{\sigma}

\newcommand\fr{\frac}
\newcommand\al{\alpha}
\newcommand{\dv}{\mathrm{div}}

\newcommand\les{\lesssim}
\newcommand\lm{\lambda}
\newcommand\Lm{\Lambda}

\newcommand\om{\omega}
\newcommand\Pe{\mathbb{P}}
\newcommand\pr{\partial}
\newcommand{\Rey}{\mathrm{Re}}
\newcommand{\We}{\mathrm{We}}
\newcommand{\ddl}{\dot{\Dl}_q}

\allowdisplaybreaks

\begin{document}
\title[Oldroyd-B model with a class of large initial data]{Global solutions to the Oldroyd-B model with a class of large initial data}

\author{Daoyuan Fang, Ruizhao Zi}

\address{Department of Mathematics, Zhejiang University,
Hangzhou 310027, P. R. China}
\email{dyf@zju.edu.cn}

\address{School of Mathematics and Statistics, Central
China Normal University, Wuhan 430079, P. R. China}
\email{rzz@mail.ccnu.edu.cn}



\begin{abstract}
Consider a global wellposed problem for the incompressible
Oldroyd-B model. It is shown that this set of equations admits a unique global solution provided the initial horizontal velocity $u^h_0$, the product $\om u^d_0$ of the coupling parameter $\om$ and initial the vertical velocity $u^d_0$, and initial symmetric tensor of constrains $\tau_0$ are sufficient small in the scaling invariant
Besov space $\dot{B}^{\fr{d}{2}-1}_{2,1}\times\dot{B}^{\fr{d}{2}}_{2,1}, d\ge2$.
In particular, the result implies the global well-posedness
of Oldroyd-B model with large initial vertical velocity $u_0^d$.
\end{abstract}
\maketitle

\section{Introduction}
\noindent In this paper, we focus on a prototypical model for viscoelastic fluids, the so called Oldroyd-B
model, which was originally  introduced by  J.G. Oldroyd
\cite{Oldroyd58}.  This type of fluids is  described by the following set of equations
\begin{eqnarray}\label{IOBdimensionless}
\begin{cases}
\mathrm{Re}\left(u_t+(u\cdot\nabla) u\right)-(1-\om)\Delta u+\nabla
\Pi=\dv\tau,\\
\mathrm{We}(\tau_t+(u\cdot\nabla)\tau+g_\al(\tau,
\nabla u))+\tau=2\om D(u), \\
\dv u=0,
\end{cases}
\end{eqnarray}
where $u$ is the velocity of the fluid, $\tau$ is the symmetric tensor of constrains,  and $\Pi$ is the pressure which can be regarded as the Lagrange multiplier for the divergence free condition. $\mathrm{Re}$ and  $\mathrm{We}$ denote
the Reynolds and Weissenberg number of the fluid, respectively. $\om\in(0,1)$ is the coupling parameter. Moreover, the quadratic term
$g_\al$ is given by $g_\al(\tau, \nabla u):=\tau W(u)-W(u)\tau-\al\left(D(u)\tau+\tau D(u)\right)$
for some $\al\in[-1,1]$, where $D(u)=1/2(\nabla u+(\nabla u)^\top)$ is the
deformation tensor and $W(u)=1/2(\nabla u-(\nabla u)^\top)$ is the vorticity tensor.

A brief derivation of system \eqref{IOBdimensionless} is as follows.
In fact, the dynamics of homogeneous, isothermal, and incompressible fluid  is described by
\begin{eqnarray}\label{velocity}
\begin{cases}
\begin{array}{rll}
\partial_t u+(u\cdot\nabla)u&=&\dv \mathbb{T},\\
\dv u&=&0,
\end{array}
\end{cases}
\end{eqnarray}
where $u, \mathbb{T}$ are the velocity and stress tensor, respectively. Moreover, $\mathbb{T}$ can be
decomposed into $-p\, \mathrm{Id}+\tau$, where $p$ denotes the pressure of the fluid and $\tau$  the
tangential part of the stress tensor, respectively. For the Oldroyd-B model, $\tau$ is given by
the relation
\begin{equation}\label{tauequation}
\tau+\lambda\frac{\mathcal{D}_\al\tau}{\mathcal{D}t}=2\eta\left(D(u)+\mu\frac{\mathcal{D}_\al D(u)}{\mathcal{D}t}\right),
\end{equation}
where $\frac{\mathcal{D}_\al}{\mathcal{D}t}$ denotes the ``objective derivative"
$$
\frac{\mathcal{D}_\al\tau}{\mathcal{D}t}=\partial_t\tau+(u\cdot\nabla)\tau+g_\al(\tau,\nabla u).
$$
The parameters
$\lambda>\mu>0$ denote  the relaxation and retardation time, respectively, and $\eta$ is the viscosity of the fluid.

The tangential part of the stress tensor $\tau$ can be decomposed as $\tau=\tau_N+\tau_e$
where $\tau_N=2\eta\frac{\mu}{\lambda}D(u)$ corresponds to the Newtonian part and
$\tau_e$ to the purely elastic part.
Denoting $\tau_e=\tau$, the above equations can be rewritten as
\begin{eqnarray}\label{utau}
\begin{cases}
\begin{array}{rll}
\partial_t u+(u\cdot\nabla)u-\eta(1-\om)\Delta u+\nabla p&=&\dv\tau,\\
\dv u&=&0,\\
\lambda(\pr_t\tau+(u\cdot\nabla)\tau+g_\al(\tau,
\nabla u))+\tau&=&2\eta\om D(u),
\end{array}
\end{cases}
\end{eqnarray}
with the coupling parameter $\om:=1-\frac{\mu}{\lambda}\in (0,1)$. Using dimensionless variables, Oldroyd-B model can be thus described by the system
\eqref{IOBdimensionless}. For more details of the modeling,  c.f., \cite{Talhouk94,FZ14} for example.

The theory of Oldroyd-B fluids recently gained quite some attention. According to the coupling strength  between the velocity $u$ and the symmetric tensor of the constrains $\tau$, the results on Oldroyd-B model fall into two categories. \par
\noindent{\bf (I) {\em Weak } coupling case: small coupling parameter $\om$.}
The study  of the incompressible Oldroyd-B model was started by a pioneering
paper given by Guillop\'e and Saut \cite{GS90}. They proved that (i) system \eqref{IOBdimensionless} admits a unique local strong solution in
suitable Sobolev spaces $H^s(\Omega)$ for  {\em bounded domains} $\Omega \subset \R^3$; and (ii) this solution
is global provided the data as well as the coupling
parameter $\om$  are sufficiently small.
For extensions to this results to the $L^p$-setting, see the work of
Fernand\'ez-Cara, Guill\'en and Ortega \cite{FGG98}. The situation of {\em exterior domains} was considered
first by  Hieber, Naito and Shibata \cite{HNS12}, where the existence of a unique global strong
solution defined in certain function spaces was proved provided the
initial data and the coupling parameter $\om$ are small enough.
Another remarkable  way to study the Oldroyd-B model \eqref{IOBdimensionless} is constructing solutions in  {\em scaling invariant} spaces. This approach goes back to the seminal work by Fujita and Kato \cite{FK64} for the classical incompressible Navier-Stokes equations.  While  for the Oldroyd-B model, strictly speaking, the system \eqref{IOBdimensionless} does not have any scaling invariance. Fortunately, if we neglect the coupling term $\dv \tau$ and the damping term $\tau$, it is found that \eqref{IOBdimensionless} is invariant under the transformation
\begin{gather}\label{scaling}
\begin{cases}
(u_0, \tau_0)\rightarrow (\ell u_0(\ell x),\tau_0(\ell x)),\\
( u(t,x), \tau(t,x),\Pi(t,x))\rightarrow (\ell u(\ell x,\ell^2t), \tau(\ell x,\ell^2t), \ell^2\Pi(\ell x,\ell^2t)),
\end{cases}
\end{gather}
for any $\ell>0$.
Chemin and Masmodi \cite{CM01} first study the Oldroyd-B model \eqref{IOBdimensionless}  alone this line. For small coupling parameter $\om$, they showed the
global well-posedness of system \eqref{IOBdimensionless} with initial data $(u_0,\tau_0)$ belonging to the critical space
$\left(\dot{B}^{\fr{d}{p}-1}_{p,1}\right)^d\times \left(\dot{B}^{\fr{d}{p}}_{p,1}\right)^{d\times d}$ for any $p\in[1,\infty)$.

 \noindent{\bf (II) {\em Strong } coupling case: arbitrary coupling parameter $\om\in(0,1)$.} Oldroyd-B model with strong coupling parameter describes the behavior of the viscoelastic fluids  more appropriately.
The first result on the system \eqref{IOBdimensionless} with strong coupling parameter was obtained by  Molinet and Talhouk \cite{MT04} for {\em bounded domains}.
They removed
the smallness restriction on the coupling parameter $\om$ in
\cite{GS90}. For the situation of {\em exterior domains},   Hieber and the authors recently \cite{FHZ13} improved
the main result given in \cite{HNS12} to the situation of non-small coupling parameter. Constructing solutions in {\em scaling invariant} spaces with {\em strong} coupling is more difficult. Some efforts were made by Chen and Miao in \cite{Chen-Miao08}, where they established global solutions to  the incompressible Oldroyd-B model with small initial data in
$B^s_{2,\infty}, s>\frac{d}{2}$. Later on, this result was generalized by  the second author \cite{Zi13} with  data in $\dot{H}^s\cap\dot{B}^{\fr{d}{2}}_{2,1}, -\fr{d}{2}<s<\fr{d}{2}-1$.  It was not until very recently that  Zhang and the authors \cite{Zi-Fang-Zhang14} obtained the scaling invariant solution with  strong coupling in the critical $L^p$  framework. For the {\em weak solutions} of incompressible Oldroyd-B fluids,  see the work of   Lions and Masmoudi \cite{LM00} for the case $\al=0$. The general case $\al\neq0$ is still open up to now. As for  the {\em blow-up criterions} of  the incompressible Oldroyd-B model,
there are works  \cite{CM01,KMT08,LMZ10}.  Besides, we would like to mention that
 Constantin and Kliegl \cite{CK12} proved the global regularity of solutions in two dimensional case for the Oldroyd-B fluids with {\em diffusive stress}.    An approach based on  {\em Lagrangian particle dynamics} can be found in \cite{Lei10,Lei07,LLZ08,LZ05,LLZ05,LZ08,Qian10,Zhang-Fang12}.

On the other hand, for the classical Navier-Stokes equations, taking advantage of the the algebraical structure of the incompressible condition $\dv u=0$, some recent progresses have been made such as \cite{Zhang09,PZ11, CG10, CGP11}. Following these ideas,  Zhang \cite{Zhang14} proved the existence and uniqueness of  solutions to system related to the incompressible viscoelastic fluids with
a class of large initial data.

Motivated by \cite{Zhang09, Zhang14}, the aim of this paper is to  investigate the global well-posedness of system \eqref{IOBdimensionless} in critical Besov spaces with the vertical component of the initial velocity field being
large. Basically it is the algebraical structure of system \eqref{IOBdimensionless} that enables us to consider this type of initial data. More precisely, using the divergence free condition $\dv u=0$, the vertical component of the convection term $u\cdot\nb u^d$ can be rewritten as
$u^h\pr_hu^d-u^d\dv_hu^h$. As a result, the
vertical component $u^d$ of the velocity satisfies a linear equation with coefficients depending on the horizontal
components $u^h$
and $\tau$. This is the reason why we do not believe we have to impose the smallness condition on the initial vertical velocity $u^d_0$. While the equation for the horizontal velocity $u^h$ contains bilinear terms
in the horizontal components and also terms describing the interactions between the
horizontal components and the vertical one. This necessitates the smallness restriction on the initial horizontal velocity $u^h_0$. Besides, since the equation of the symmetric tensor of constrains $\tau$ is nonlinear, considering the interactions between the velocity $u$ (with large vertical component) and $\tau$, imposing a smallness condition on $\tau_0$ is necessary. In addition, we would like to emphasize that both the equation of vertical velocity and horizontal velocity are linearly  coupled with $\tau$, and conversely, the equation of the symmetric tensor of constrains $\tau$ is linearly coupled with $u$ as well. In particular, the coupling parameter $\om$ is involved. In order to decouple $u$ and $\tau$ linearly, one  way to do is simply to apply the estimates for {\em heat equation} and {\em transport equation with damping} to the velocity $u$ and the symmetric tensor of constrains $\tau$ respectively, and to substitute the estimate for $\tau$ to that for $u$. Nevertheless, in doing this, the coupling parameter $\om$ must be small enough to close the estimate of $u$.  This  less interesting {\em weak} coupling case is exactly what we want to avoid considering.

Although the {\em strong} coupling case with initial data in $L^p$ type critical Besov space was investigated by us in \cite{Zi-Fang-Zhang14}, large initial vertical velocity is not allowed in \cite{Zi-Fang-Zhang14}. A natural question is that can we solve system \eqref{IOBdimensionless} with {\em strong} coupling parameter $\om$ and {\em large} initial vertical velocity $u^d_0$? Before answering it, let us give some heuristic observations. Given large initial vertical velocity, it has been pointed out above that the horizontal velocity $u^h$ and the symmetric tensor of constrains $\tau$ should be small. In the circumstances, the linear term $2\om D(u)$ should also be small otherwise  the equation $\eqref{IOBdimensionless}_2$ of $\tau$ can not hold. Since we can not expect the vertical velocity to be small, it is reasonable to construction solutions to system \eqref{IOBdimensionless} with the product $\om u^d$ of the coupling parameter $\om$ and the vertical velocity $u^d$ small. That's why we need the initial product $\om u^d_0$ to be small in this paper.

Let us now explain the main ingredients of the proof. Noting that the algebraical structure of system \eqref{IOBdimensionless} plays a key role in our result, the method used in \cite{Zi-Fang-Zhang14} which relied heavily on the Green matrix of the linearized system of \eqref{IOBdimensionless} and was carried our by Lagrangian approach fails to work here. It is because the divergence free  condition does not hold anymore in Lagrangian coordinate. In this paper, we exploit energy approach in the critical $L^2$ framework.
 As  the damping $\tau$ in the equation $\eqref{IOBdimensionless}_2$ may mislead us to the {\em weak} coupling case, so a crucial step would be to find a structural mechanism which indicates a parabolic smoothing effect on $u$ in the absence of the damping effect on $\tau$. To the best of our knowledge, in $L^2$ framework, researchers \cite{MT04,FHZ13,Chen-Miao08,Zi13} were inclined to use the cancelation relation
  \be\label{cancelation}
  (\dv\tau|u)+(D(u)|\tau)=0
  \ee
 to eliminate the effect of the coupling parameter $\om$. The price they had to pay was  that the solution space served for $u$ and $\tau$ would be of the same regularity, which violated the scaling. So our strategy is not to estimate $(u,\tau)$ as a whole with an eye to \eqref{cancelation}, but to find a new quantity which obeys the scaling and can be used as a substitute for $\tau$. An ideal candidate is $\Pe\dv\tau$ because on the one hand, $\Pe\dv\tau$ is the exact quantity  linearly coupled with $u$ and obeying the scaling. On the other hand, it has been shown in \cite{Zi-Fang-Zhang14} that the Green matrix $G(t,x)$ of  the following auxiliary system of $(u, \Pe\dv\tau)$
\beq\label{auxiliary-liOB}
\begin{cases}
\begin{array}{rrl}
\mathrm{Re}\,u_t-(1-\om)\Delta u-\Pe\dv\tau&=&-\mathrm{Re}\,\Pe(u\cdot\nb)u,\\
\mathrm{We}\,\Pe\dv\tau_t+\Pe\dv\tau-\om \Dl u&=&-\mathrm{We}\,\left(\Pe\dv(u\cdot\nb)\tau+\Pe\dv g_\al(\tau,\nb u)\right), \\
u(0) &=&u_0, \\
\Pe\dv \tau(0)&=&\Pe\dv \tau_0.
\end{array}
\end{cases}
\eeq
behaves just like the heat kernel in the low frequency case. In this paper, we will pay much attention to the system \eqref{auxiliary-liOB}  in the spirit of Danchin \cite{Danchin00}. More precisely, we first localize \eqref{auxiliary-liOB}  in frequencies according
to the Littlewood-Paley decomposition. Then a variant energy method will be used  to
estimate each dyadic block $(\ddl u,\ddl \Pe\dv\tau)$. Owing to the fact that the linear operator associated with \eqref{auxiliary-liOB} can not be diagonalized in a basis independent of $\xi$,  coercive estimates can not be achieved by standard energy argument. Since the parabolic-hyperbolic system \eqref{auxiliary-liOB} behaves differently in high and low frequencies, we shall consider the following quantity (for some $q_0, q_1\in\Z$ to be chosen later):
\beqno\label{Y}
Y_q=\begin{cases}
\begin{array}{lrl}
\left[\|u_q\|_{L^2}^2+2\left\|\fr{\We}{\Rey}\Pe\dv\tau_q\right\|_{L^2}^2+2(u_q|\fr{\We}{\Rey}\Pe\dv\tau_q)+2( u_q|\fr{\We}{\Rey}\Dl\Pe\dv\tau_q)\right]^{\fr12}, & \mathrm{for}& q\le q_1,\\[3mm]

\left[2\| u_q\|_{L^2}^2+\left\|\fr{(1-\om)\We}{\om\Rey}\Pe\dv\tau_q\right\|_{L^2}^2-2\left( u_q|\fr{(1-\om)\We}{\om\Rey}\Pe\dv\tau_q\right)\right]^{\fr12}, & \mathrm{for}& q>q_0,\\[3mm]

\left[\| u_q\|_{L^2}^2+\fr{\We}{\om\Rey}\|\Lm^{-1}\Pe\dv\tau_q\|_{L^2}^2
\right]^{\fr12}, & \mathrm{for}& q_1<q\le q_0.
\end{array}
\end{cases}
\eeqno
Some words about the definition of $Y_q$ are in order.
\begin{itemize}
\item For $q\le q_1$, the parabolic smoothing effect of $\Pe\dv\tau$ is obtained with the aid of the linear coupling term $-\Pe\dv\tau$ in $\eqref{auxiliary-liOB}_1$, that's why we introduce the cross term $( u_q|\fr{\We}{\Rey}\Dl\Pe\dv\tau_q)$.

\item For $q>q_0$, compared with the linear coupling term $-\Pe\dv\tau$ in $\eqref{auxiliary-liOB}_1$,  the damping term $\Pe\dv\tau$ in $\eqref{auxiliary-liOB}_2$ is predominate, and thus there is a damping effect on $\Pe\dv\tau$.
\item For $q_1<q\le q_0$, a new quantity $\Lm^{-1}\Pe\dv\tau$ is introduced to fill up the gap between $q_1$ and $q_0$. Indeed, in this case, we do not need to care about the scaling thanks to Bernstein's inequalities. On this basis, the cancelation relation \eqref{cancelation} motivates us to choose  $\Lm^{-1}\Pe\dv\tau$ to replace $\Pe\dv\tau$. In doing so, $u$ and $\Lm^{-1}\Pe\dv\tau$ are linearly decoupled, and $\Lm^{-1}\Pe\dv\tau$  also possesses a damping effect because of the damping term. Therefore, in this paper we combine the case $q>q_0$ with $q_1<q\le q_0$, and regard them as high frequency. The left part $q\le q_1$ is regarded as low frequency.
\end{itemize}
We would like to remark that similar construction of $Y_q$ was also used by the second author in \cite{Zi14} for the compressible Oldroyd-B model.
However, for $q_1<q\le q_0$, the quantity in \cite{Zi14} corresponding to $Y_q$ is
\beqno
\tl{Y}_q:=\left[\|\Pe u_q\|_{L^2}^2+\fr{\We}{\om\Rey}\|\Lm^{-1}\Pe\dv\tau_q\|_{L^2}^2
-\fr{2\beta(1-\om)\We}{\om\Rey}(\Pe u_q|\Pe\dv\tau_q)\right]^{\fr12},
\eeqno
with some positive constant $\beta$. Since the cross term $(\Pe u_q|\Pe\dv\tau_q)$ is involved in $\tl{Y}_q$, $\Lm^{-1}\Pe\dv\tau$ in \cite{Zi14} exhibits the parabolic smoothing effect stemming from the linear coupling term $-\Pe\dv\tau$. Hence the case $q_1<q\le q_0$ is regarded as low frequency in \cite{Zi14}.  The point is, though, that $\tl{Y}_q$ can not be used in this paper to replace $Y_q$ because the parabolic smoothing effect for  $\Lm^{-1}\Pe\dv\tau$ and hence for $u$ is too weak (depending on $\om$) to deal with the linear coupling term $2\om D(u)$ in $\eqref{IOBdimensionless}_2$.

Obviously, the above analysis do not take into account the anisotropy. In order to distinguish  the velocity $u$ from different directions,  we must bound $(u^h,(\Pe\dv\tau)^h)$ and $(u^d,(\Pe\dv\tau)^d)$ separately (please refer to Section \ref{sec-apriori} for more details).

We shall obtain the existence and uniqueness of a solution $(u,\tau)$ to \eqref{IOBdimensionless} in the following space.
\begin{defn}\label{space}
For $T>0$,  and $s\in\R$, let us denote
\beno
\mathcal{E}^s_T:=\left(\tl{C}_T(\dot{B}^{s-1}_{2,1})\cap L^1_T(\dot{B}^{s+1}_{2,1})\right)^d\times\left(\tl{C}_T(\dot{B}^{s}_{2,1})\cap L^1_T(\dot{B}^{s}_{2,1})\right)^{d\times d}.
\eeno
We use the notation $\mathcal{E}^s$ if $T=\infty$, changing $[0, T]$ into $[0,\infty)$ in the definition above.
\end{defn}

Our main result is stated as follows:
\begin{thm}\label{thm-g}
Let $d\ge2$. Assume that $(u_0,\tau_0)\in \left(\dot{B}^{\fr{d}{2}-1}_{2,1}\right)^d\times \left(\dot{B}^{\fr{d}{2}}_{2,1}\right)^{d\times d}$ with $\dv u_0=0$.
There exists a positive constant $M$, depending on $d,  \Rey$ and $\We$, but independent of $\om$,such that if
                                      \be\label{data}
                                      \fr{C_0}{(1-\om)^9}\left(\|u^h_0\|_{\dot{B}^{\fr{d}{2}-1}_{2,1}}+\om\|u^d_0\|_{\dot{B}^{\fr{d}{2}-1}_{2,1}}
+\|\tau_0\|_{\dot{B}^{\fr{d}{2}-1}_{2,1}}\right)\exp\left(\fr{C_0}{(1-\om)^{12}}\|u^d_0\|_{\dot{B}^{\fr{d}{2}-1}_{2,1}}
^2\right)\le 1,
                                      \ee
then system \eqref{IOBdimensionless} admits a unique global solution $(u,\tau)$ in $\mathcal {E}^{\fr{d}{2}}$ with
\beq
\nn&&\|u^d\|_{\tl{L}^\infty_t(\dot{B}^{\fr{d}{2}-1}_{2,1})}+\|(\Pe\dv\tau)^d\|_{\tl{L}^\infty_t(\dot{B}^{\fr{d}{2}-1}_{2,1})}
+\|u^d\|_{L^1_t(\dot{B}^{\fr{d}{2}+1}_{2,1})}+\|(\Pe\dv\tau)^d\|_{L^1_t(\dot{B}^{\fr{d}{2}+1,\fr{d}{2}-1}_{2,1})}\\
&\le&\fr{C_0}{(1-\om)^3}\left(1+\|u^d_0\|_{\dot{B}^{\fr{d}{2}-1}_{2,1}}
+\|(\Pe\dv\tau)^d_0\|_{\dot{B}^{\fr{d}{2}-1}_{2,1}}\right),
\eeq
and
\beq
\nn&&\|u^h\|_{\tl{L}^\infty_t(\dot{B}^{\fr{d}{2}-1)}_{2,1}}+\|(\Pe\dv\tau)^h\|_{\tl{L}^\infty_t(\dot{B}^{\fr{d}{2}-1}_{2,1})}
+\|u^h\|_{L^1_t(\dot{B}^{\fr{d}{2}+1}_{2,1})}+\|(\Pe\dv\tau)^h\|_{L^1_t(\dot{B}^{\fr{d}{2}+1,\fr{d}{2}-1}_{2,1})}\\
&\le&\fr{C_0}{(1-\om)^6}\left(\|u^h_0\|_{\dot{B}^{\fr{d}{2}-1}_{2,1}}+\om\|u^d_0\|_{\dot{B}^{\fr{d}{2}-1}_{2,1}}
+\|\tau_0\|_{\dot{B}^{\fr{d}{2}-1}_{2,1}}\right)\exp\left(\fr{C_0}{(1-\om)^{12}}\|u^d_0\|_{\dot{B}^{\fr{d}{2}-1}_{2,1}}
^2\right),
\eeq
and
\beq
\nn&&\|\tau\|_{\tl{L}^\infty_t(\dot{B}^{\fr{d}{2}}_{2,1})}+\|\tau\|_{L^1_t(\dot{B}^{\fr{d}{2}}_{2,1})}\\
&\le&\fr{C_0}{(1-\om)^6}\left(\|u^h_0\|_{\dot{B}^{\fr{d}{2}-1}_{2,1}}+\om\|u^d_0\|_{\dot{B}^{\fr{d}{2}-1}_{2,1}}
+\|\tau_0\|_{\dot{B}^{\fr{d}{2}-1}_{2,1}}\right)\exp\left(\fr{C_0}{(1-\om)^{12}}\|u^d_0\|_{\dot{B}^{\fr{d}{2}-1}_{2,1}}
^2\right).
\eeq
\end{thm}

The rest part of this paper is organized as follows. In Section 2, we introduce the tools ( the Littlewood-Paley decomposition
and paradifferertial calculus) and give some nonlinear estimates in Besov spaces.  Section 3 is devoted to the global a priori estimates of system \eqref{IOBdimensionless}. The proof of Theorem \ref{thm-g} is given in Section 4.

\bigbreak\noindent{\bf Notation.}\par
\begin{enumerate}
\item For $f\in \mathcal{S}'$, $f_q:=\ddl f$.
\item For $x=(x_1,\cdots,x_d)$, we denote $x_h:=(x_1,\cdots,x_{d-1})$; For vector field $u(x)=(u^1(x),  \cdots, u^d(x)),$ we denote $u^h(x):=(u^1(x),  \cdots, u^{d-1}(x))$, $\pr_hu(x):=(\pr_1u(x),  \cdots, \pr_{d-1}u(x))^\top$, and $\dv_hu^h(x):=\pr_1u^1(x)+\cdots+ \pr_{d-1}u^{d-1}(x)$.
\item For $a, b\in L^2$, $(a|b)$ denotes the $L^2$ inner product of $a$ and $b$.

\item Throughout the paper, $C$ denotes various ``harmless'' positive constants, and
we sometimes use the notation $A \lesssim B$ as an equivalent to $A \le CB$. The
notation $A \approx B$ means that $A \lesssim B$ and $B \lesssim A$.
\end{enumerate}
\section{The Functional Tool Box}
\noindent The results of the present paper rely on the use of  a
dyadic partition of unity with respect to the Fourier variables, the so-called the
\textit{Littlewood-Paley  decomposition}. Let us briefly explain how
it may be built in the case $x\in \R^d$ which the readers may see more details
in \cite{Bahouri-Chemin-Danchin11,Ch1}. Let $(\chi, \varphi)$ be a couple of $C^\infty$ functions satisfying
\be\label{support}
\hbox{Supp}\chi\subset\{|\xi|\leq\frac{4}{3}\},
\quad
\hbox{Supp}\varphi\subset\{\frac{3}{4}\leq|\xi|\leq\frac{8}{3}\},
\ee
and
$$\chi(\xi)+\sum_{q\geq0}\varphi(2^{-q}\xi)=1,$$

$$\sum_{q\in \mathbb{Z}}\varphi(2^{-q}\xi)=1, \quad \textrm{for} \quad \xi\neq0.$$
Set $\varphi_q(\xi)=\varphi(2^{-q}\xi),$
$h_q=\mathcal{F}^{-1}(\varphi_q),$ and
$\tilde{h}=\mathcal{F}^{-1}(\chi)$. The dyadic blocks and the low-frequency cutoff operators are defined for all $q\in\mathbb{Z}$ by
$$\dot{\Delta}_{q}u=\varphi(2^{-q}\mathrm{D})u=\int_{\R^d}h_q(y)u(x-y)dy,$$
$$\dot{S}_qu=\chi(2^{-q}\mathrm{D})u=\int_{\R^d}\tl{h}_q(y)u(x-y)dy.$$
Then
\begin{equation}\label{e2.1}
u=\sum_{q\in \mathbb{Z}}\Delta_qu,
\end{equation}
holds for tempered distributions {\em modulo polynomials}. As working modulo polynomials is not appropriate for nonlinear problems, we
shall restrict our attention to the set $\mathcal {S}'_h$ of tempered distributions $u$ such that
$$
\lim_{q\rightarrow-\infty}\|\dot{S}_qu\|_{L^\infty}=0.
$$
Note that \eqref{e2.1} holds true whenever $u$ is in $\mathcal{S}'_h$ and that one may write
$$
\dot{S}_qu=\sum_{p\leq q-1}\dot{\Dl}_{p}u.
$$
Besides, we would like to mention that the Littlewood-Paley decomposition
has a nice property of quasi-orthogonality:
\begin{equation}\label{e2.2}
\dot{\Delta}_p\dot{\Delta}_qu\equiv 0\ \ \hbox{if}\ \ \ |p-q|\geq 2\ \
\hbox{and}\ \ \dot{\Delta}_p(\dot{S}_{q-1}u\dot{\Delta}_qu)\equiv 0\ \ \hbox{if}\ \ \
|p-q|\geq 5.
\end{equation}
One can now  give the definition of
homogeneous Besov spaces.
\begin{defn}\label{D2.1}
For $s\in\R$, $(p,r)\in[1,\infty]^2$, and
$u\in\mathcal{S}'(\R^d),$ we set
$$\|u\|_{\dot{B}_{p,r}^s}=\left\|2^{ sq}\|\dot{\Delta}_qu\|_{L^p} \right\|_{\ell^r}.$$
We then define the space
$\dot{B}_{p,r}^s:=\{u\in\mathcal{S}'_h(\R^d),\
\|u\|_{\dot{B}_{p,r}^s}<\infty\}$.
\end{defn}
Since homogeneous Besov spaces fail to have nice inclusion properties, it is wise to define {\em hybrid Besov spaces} where the growth conditions satisfied by the dyadic blocks are different for low and high frequencies. In fact,  hybrid Besov spaces played a crucial role for proving global well-posedness of compressible barotropic Navier-Stokes equations in critical spaces \cite{Danchin00}. Let us now define the hybrid Besov spaces that we need. Here our notations are somehow different from those in \cite{Danchin00}.
\begin{defn}\label{def-hybrid}
Let $s, t\in\R$, and
$u\in\mathcal{S}'(\R^d)$. For some fixed $q_0\in\Z$, we set
\beno
\|u\|_{\dot{B}^{s,t}_{2,1}}:=\sum_{q\le q_0}2^{qs}\|\ddl u\|_{L^2}+\sum_{q>q_0}2^{qt}\|\ddl u\|_{L^2}.
\eeno
We then define the space
$\dot{B}_{2,1}^{s,t}:=\{u\in\mathcal{S}'_h(\R^d),\
\|u\|_{\dot{B}_{2,1}^{s,t}}<\infty\}$.
\end{defn}
\begin{rem}
For all $s, t\in\R, \tl{q}_0\in\Z$, and
$u\in\mathcal{S}'(\R^d)$, setting
\beno
\|u\|_{\tl{B}^{s,t}_{2,1}}:=\sum_{q\le \tl{q}_0}2^{qs}\|\ddl u\|_{L^2}+\sum_{q>\tl{q}_0}2^{qt}\|\ddl u\|_{L^2},
\eeno
then it is easy to verify that $\|u\|_{\tl{B}^{s,t}_{2,1}}\approx\|u\|_{\dot{B}^{s,t}_{2,1}}$.
\end{rem}
The following lemma describes the way derivatives act on spectrally localized functions.
\begin{lem}[Bernstein's inequalities]\label{Bernstein}
Let $k\in\N$ and $0<r<R$. There exists a constant $C$ depending on $r, R$ and $d$ such that for all $(a,b)\in[1,\infty]^2$, we have for all $\lm>0$ and multi-index $\al$
\begin{itemize}
\item If $\mathrm{Supp} \hat{f}\subset B(0,\lm R)$, then $\sup_{\al=k}\|\pr^\al f\|_{L^b}\le C^{k+1}\lm^{k+d(\fr1a-\fr1b)}\|f\|_{L^a}$.
\item If $\mathrm{Supp} \hat{f}\subset \mathcal{C}(0,\lm r, \lm R)$, then $C^{-k-1}\lm^k\|f\|_{L^a}\le\sup_{|\al|=k}\|\pr^\al f\|_{L^a}\le C^{k+1}\lm^k\|f\|_{L^a}$
\end{itemize}
\end{lem}
\noindent In fact,  the following special case of Lemma \ref{Bernstein} will be used very frequently in this paper.
\begin{coro}\label{coro-Bernstein}
Let $\Lm:=\sqrt{-\Dl}$, then for any $f\in\mathcal {S}'$, there holds
\be\label{equivalent}
\fr342^q\|\ddl f\|_{L^2}\le\|\Lm \ddl f\|_{L^2}\le\fr832^q\|\ddl f\|_{L^2}.
\ee
\end{coro}
\noindent Owing to the property \eqref{support} of  $\mathrm{Supp}\varphi$, \eqref{equivalent} can be easily obtained by virtue of Plancherel's Theorem.

Let us now state some classical
properties for the Besov spaces.
\begin{prop}\label{prop-classical}
For all $s, s_1, s_2\in\R$, $1\le p, p_1, p_2, r, r_1, r_2\le\infty$, the following properties hold true:

\begin{itemize}
\item If $p_1\leq p_2 $  and $r_1\leq r_2,$ then
$\dot{B}_{p_1,r_1}^{s}\hookrightarrow
\dot{B}_{p_2,r_2}^{s-\frac{d}{p_1}+\frac{d}{p_2}}$.

\item If $s_1\neq s_2$ and $\theta\in(0,1)$,
$\left[\dot{B}_{p,r_1}^{s_1},\dot{B}_{p,r_2}^{s_2}\right]_{(\theta,r)}=\dot{B}_{p,r}^{\theta
s_1+(1-\theta)s_2}.$

\item For any smooth homogeneous of degree $m\in\Z$ function $F$ on $\R^d\backslash\{0\}$, the operator $F(D)$ maps $\dot{B}^{s}_{p,r}$ in $\dot{B}^{s-m}_{p,r}$.
\end{itemize}
\end{prop}

Next we  recall a few nonlinear estimates in Besov spaces which may be
obtained by means of paradifferential calculus. Firstly introduced
 by J. M. Bony in \cite{Bony81}, the paraproduct between $f$
and $g$ is defined by
$$\dot{T}_fg=\sum_{q\in\mathbb{Z}}\dot{S}_{q-1}f\dot{\Delta}_qg,$$
and the remainder is given by
$$\dot{R}(f,g)=\sum_{q\geq -1}\dot{\Delta}_qf\tilde{\dot{\Delta}}_qg$$
with
$$\tilde{\dot{\Delta}}_qg=(\dot{\Delta}_{q-1}+\dot{\Delta}_{q}+\dot{\Delta}_{q+1})g.$$
We have the following so-called Bony's decomposition:
 \be\label{Bony-decom}
fg=\dot{T}_fg+\dot{T}_gf+\dot{R}(f,g)=\dot{T}_fg+\dot{T}'_gf,
 \ee
where $\dot{T}'_gf:=\dot{T}_gf+\dot{R}(f,g)$. The paraproduct $\dot{T}$ and the remainder $\dot{R}$ operators satisfy the following
continuous properties.

\begin{prop}\label{p-TR}
For all $s\in\mathbb{R}$, $\sigma>0$, and $1\leq p, p_1, p_2, r, r_1, r_2\leq\infty,$ the
paraproduct $\dot T$ is a bilinear, continuous operator from $L^{\infty}\times \dot{B}_{p,r}^s$ to $
\dot{B}_{p,r}^{s}$ and from $\dot{B}_{\infty,r_1}^{-\sigma}\times \dot{B}_{p,r_2}^s$ to
$\dot{B}_{p,r}^{s-\sigma}$  with $\frac{1}{r}=\min\{1, \frac{1}{r_1}+\frac{1}{r_2}\}$. The remainder $\dot R$ is bilinear continuous from
$\dot{B}_{p_1,r_1}^{s_1}\times \dot{B}_{p_2,r_2}^{s_2}$ to $
\dot{B}_{p,r}^{s_1+s_2}$ with
$s_1+s_2>0$,  $\frac{1}{p}=\frac{1}{p_1}+\frac{1}{p_2}\leq1$, and $\frac{1}{r}=\frac{1}{r_1}+\frac{1}{r_2}\leq1$.
\end{prop}
In view of \eqref{Bony-decom}, Proposition \ref{p-TR} and Bernstein's inequalities,  one easily deduces the following  product estimates:
\begin{coro}\label{coro-product}
Let $p\in[1,\infty]$. If $s_1, s_2\le \fr{d}{p}$ and $s_1+s_2>d\max \{0,\fr{2}{p}-1\}$, then there holds
\be\label{product1}
\|uv\|_{\dot{B}^{s_1+s_2-\frac{d}{p}}_{p,1}}\leq C\|u\|_{\dot{B}^{s_1}_{p,1}}\|v\|_{\dot{B}^{s_2}_{p,1}}.
\ee
\end{coro}
In the following, we shall give a commutator estimate, which will be used to deal with the convection terms.
\begin{lem}[\cite{Zi14}]\label{lem-commu}
Let $\Lm:=\sqrt{-\Dl}$, then
\beq\label{commu}
\|[\Lm^{-1}, \dot{S}_{q-1}v\cdot\nb]\ddl u\|_{L^2}\le C\|\nb \dot{S}_{q-1}v\|_{L^\infty}\|\Lm^{-1}\ddl u\|_{L^2}.
\eeq
\end{lem}

The divergence free condition allows us to get the following lemma, which plays a key role in this paper.
\begin{lem}[\cite{Zhang14}]\label{d-f}
Let $v$ be a divergence free vector field. We then have
 \beno
\|v^d\|^2_{L^2_{x_h}(L^\infty_{x_d})}\le C\|\dv_h v^h\|_{L^2}\|v^d\|_{L^2},
\quad
\mathrm{and}\quad
\|v^d\|_{L^\infty}\le C\|v^h\|^{\fr12}_{\dot{B}^{\fr{d}{2}}_{2,1}}\|v^d\|^{\fr12}_{\dot{B}^{\fr{d}{2}}_{2,1}}.
\eeno
\end{lem}

The study of non-stationary PDEs requires spaces of the type
$L^\rho_T(X)=L^\rho(0,T;X)$ for appropriate Banach spaces $X$. In
our case, we expect $X$ to be a  Besov space, so that it
is natural to localize the equations through Littlewood-Paley
decomposition. We then get estimates for each dyadic block and
perform integration in time. But, in doing so, we obtain the bounds
in spaces which are not of the type $L^\rho(0,T;\dot{B}^s_{p,r})$. That
 naturally leads to the following definition introduced by Chemin and Lerner in \cite{CL}.
\begin{defn}\label{defn-chemin-lerne}
For $\rho\in[1,+\infty]$, $s\in\R$, and $T\in(0,+\infty)$, we set
$$\|u\|_{\tilde{L}^\rho_T(\dot{B}^s_{p,r})}=\left\|2^{qs}
\|\dot{\Delta}_qu(t)\|_{L^\rho_T(L^p)}\right\|_{\ell^r}
$$
and denote by
$\tilde{L}^\rho_T(\dot{B}^s_{p,r})$ the subset of distributions
$u\in\mathcal{S}'([0,T]\times \mathbb{R}^N)$ with finite
$\|u\|_{\tilde{L}^\rho_T(\dot{B}^s_{p,r})}$ norm. When $T=+\infty$, the index $T$ is
omitted. We
further denote $\tilde{C}_T(\dot{B}^s_{p,r})=C([0,T];\dot{B}^s_{p,r})\cap
\tilde{L}^\infty_{T}(\dot{B}^s_{p,r}) $.
 \end{defn}
\begin{rem}\label{rem-CM-holder}
All the properties of continuity for the paraproduct, remainder, and product remain true for the Chemin-Lerner spaces. The exponent $\rho$ just has to behave according to H\"{o}lder's ineauality for the time variable.
\end{rem}

\begin{rem}\label{rem-CM-minkowski}
The spaces $\tl{L}^\rho_T(\dot{B}^s_{p,r})$ can be linked with the classical space $L^\rho_T(\dot{B}^s_{p,r})$ via the Minkowski inequality:
\beno
\|u\|_{\tl{L}^\rho_T(\dot{B}^s_{p,r})}\le\|u\|_{L^\rho_T(\dot{B}^s_{p,r})}\quad \mathrm{if}\quad r\ge\rho,\qquad \|u\|_{\tl{L}^\rho_T(\dot{B}^s_{p,r})}\ge\|u\|_{L^\rho_T(\dot{B}^s_{p,r})}\quad \mathrm{if}\quad r\le\rho.
\eeno
\end{rem}
\section{{\em A Priori} Estimates}\label{sec-apriori}
\noindent We begin this section by rewriting the system \eqref{IOBdimensionless} in term of $(u^i,(\Pe\dv\tau)^i),1\le i\le d$, as follows
\beq\label{piOB}
\begin{cases}
\pr_t u^i-\frac{1-\om}{\Rey}\Dl u^i-\frac{1}{\Rey}(\Pe\dv\tau)^i=f^i,\\
\pr_t(\Pe\dv\tau)^i+\frac{1}{\We}(\Pe\dv\tau)^i-\frac{\om}{\We}\Dl u^i=h^i,
\end{cases}
\eeq
where $f^i:=-(\Pe(u\cdot\nb)u)^i$, and $h^i:=-\left(\Pe\dv(u\cdot\nb)\tau+\Pe\dv g_\al(\tau,\nb u)\right)^i$. For simplicity, let us denote $\sig:=\Pe\dv\tau$, then system \eqref{piOB} can be localized as
 \beq\label{loc-piOB}
\begin{cases}
\pr_t u^i_q+\dot{S}_{q-1}u\cdot\nb u^i_q-\frac{1-\om}{\Rey}\Dl u^i_q-\frac{1}{\Rey}\sig^i_q=\tl{f}^i_q,\\
\pr_t\sig^i_q+\dot{S}_{q-1}u\cdot\nb \sig^i_q+\frac{1}{\We}\sig^i_q-\frac{\om}{\We}\Dl u^i_q=\tl{h}^i_q,
\end{cases}
\eeq
where $\tl{f}^i_q:=\dot{S}_{q-1}u\cdot\nb u^i_q+f^i_q$, and $\tl{h}^i_q:=\dot{S}_{q-1}u\cdot\nb \sig^i_q+h^i_q$.

\noindent Denote
\beqno
A^h(t)&:=&\|u^h\|_{\tl{L}^\infty_t(\dot{B}^{\fr{d}{2}-1)}_{2,1}}+\|\sig^h\|_{\tl{L}^\infty_t(\dot{B}^{\fr{d}{2}-1}_{2,1})}+\|u^h\|_{L^1_t(\dot{B}^{\fr{d}{2}+1}_{2,1})}+\|\sig^h\|_{L^1_t(\dot{B}^{\fr{d}{2}+1,\fr{d}{2}-1}_{2,1})},\\
A^d(t)&:=&\|u^d\|_{\tl{L}^\infty_t(\dot{B}^{\fr{d}{2}-1}_{2,1})}+\|\sig^d\|_{\tl{L}^\infty_t(\dot{B}^{\fr{d}{2}-1}_{2,1})}+\|u^d\|_{L^1_t(\dot{B}^{\fr{d}{2}+1}_{2,1})}+\|\sig^d\|_{L^1_t(\dot{B}^{\fr{d}{2}+1,\fr{d}{2}-1}_{2,1})},\\
A^h(0)&:=&\| u^h_0\|_{\dot{B}^{\fr{d}{2}-1}_{2,1}}+\left\|\sig^h_0\right\|_{\dot{B}^{\fr{d}{2}-1}_{2,1}},\quad A^d(0):=\| u^d_0\|_{\dot{B}^{\fr{d}{2}-1}_{2,1}}+\left\|\sig^d_0\right\|_{\dot{B}^{\fr{d}{2}-1}_{2,1}},\\
B(t)&:=&\|\tau\|_{\tl{L}^\infty_t(\dot{B}^{\fr{d}{2}}_{2,1})}+\|\tau\|_{L^1_t(\dot{B}^{\fr{d}{2}}_{2,1})},\quad B(0):=\|\tau_0\|_{\dot{B}^{\fr{d}{2}}_{2,1}}.
\eeqno
We will establish the following proposition.
\begin{prop}\label{P1}
Let $(u, \tau)$ be a solution to system \eqref{IOBdimensionless}. There exist  constants $C_1, C_2$ and $C_3$ depending on $d, \Rey$, and $\We$ but independent of $\om$, such that the following inequalities hold
\beq\label{uh}
\nn A^h(t)&\le&\fr{C_1}{(1-\om)^3}\Bigg(A^h(0)+\int_0^t\left(\| u^h\|_{\dot{B}^{\fr{d}{2}+1}_{2,1}}+\| u^d\|_{\dot{B}^{\fr{d}{2}+1}_{2,1}}\right)\|\tau\|_{\dot{B}^{\fr{d}{2}}_{2,1}}dt'\\
\nn&&+\int_0^t\left(\|u^h\|_{\dot{B}^{\fr{d}{2}-1}_{2,1}}\|u^h\|_{\dot{B}^{\fr{d}{2}+1}_{2,1}}+\|u^h\|^{\fr12}_{\dot{B}^{\fr{d}{2}+1}_{2,1}}\|u^d\|^{\fr12}_{\dot{B}^{\fr{d}{2}-1}_{2,1}}\|u^h\|^{\fr12}_{\dot{B}^{\fr{d}{2}-1}_{2,1}}\|u^d\|^{\fr12}_{\dot{B}^{\fr{d}{2}+1}_{2,1}}\right)dt'\\
&&+\int_0^t\|u^h\|_{\dot{B}^{\fr{d}{2}+1}_{2,1}}^{\fr12}\|u^d\|_{\dot{B}^{\fr{d}{2}+1}_{2,1}}^{\fr12}\| u^h\|_{\dot{B}^{\fr{d}{2}-1}_{2,1}}dt'\Bigg),
\eeq

\beq\label{ud}
\nn A^d(t)&\le&\fr{C_2}{(1-\om)^3}\Bigg(A^d(0)+\int_0^t\left(\| u^h\|_{\dot{B}^{\fr{d}{2}+1}_{2,1}}+\| u^d\|_{\dot{B}^{\fr{d}{2}+1}_{2,1}}\right)\|\tau\|_{\dot{B}^{\fr{d}{2}}_{2,1}}dt'\\
\nn&&+\int_0^t\left(\|u^h\|_{\dot{B}^{\fr{d}{2}-1}_{2,1}}\|u^h\|_{\dot{B}^{\fr{d}{2}+1}_{2,1}}+\|u^h\|^{\fr12}_{\dot{B}^{\fr{d}{2}+1}_{2,1}}\|u^d\|^{\fr12}_{\dot{B}^{\fr{d}{2}-1}_{2,1}}\|u^h\|^{\fr12}_{\dot{B}^{\fr{d}{2}-1}_{2,1}}\|u^d\|^{\fr12}_{\dot{B}^{\fr{d}{2}+1}_{2,1}}\right)dt'\\
&&+\int_0^t\left(\|u^h\|_{\dot{B}^{\fr{d}{2}+1}_{2,1}}+\|u^h\|_{\dot{B}^{\fr{d}{2}+1}_{2,1}}^{\fr12}\|u^d\|_{\dot{B}^{\fr{d}{2}+1}_{2,1}}^{\fr12}\right)\| u^d\|_{\dot{B}^{\fr{d}{2}-1}_{2,1}}dt'\Bigg),
\eeq
and
\be\label{tau'}
 B(t)\le C_3\left(B(0)+\om\left(\|u^h\|_{L^1_t(\dot{B}^{\fr{d}{2}+1}_{2,1})}+\|u^d\|_{L^1_t(\dot{B}^{\fr{d}{2}+1}_{2,1})}\right)+\int_0^t(\|u^h\|_{\dot{B}^{\fr{d}{2}+1}_{2,1}}+\|u^d\|_{\dot{B}^{\fr{d}{2}+1}_{2,1}})\|\tau\|_{\dot{B}^\fr{d}{2}_{2,1}}dt'\right).
\ee
\end{prop}
\begin{proof}

\noindent
For $1\le i\le d$, let us denote
\beqno
{Y}^i_q:=\left[\|u^i_q\|_{L^2}^2+2\left\|\fr{\We}{\Rey}\sig^i_q\right\|_{L^2}^2+2(u^i_q|\fr{\We}{\Rey}\sig^i_q)+2( u^i_q|\fr{\We}{\Rey}\Dl\sig^i_q)\right]^{\fr12}, \quad \mathrm{for}\quad q\le q_1,
\eeqno

\beqno
{Y}^i_q:=\left[2\| u^i_q\|_{L^2}^2+\left\|\fr{(1-\om)\We}{\om\Rey}\sig^i_q\right\|_{L^2}^2-2\left( u^i_q|\fr{(1-\om)\We}{\om\Rey}\sig^i_q\right)\right]^{\fr12}, \quad \mathrm{for}\quad q>q_0,
\eeqno

\beqno
{Y}^i_q:=\left[\| u^i_q\|_{L^2}^2+\fr{\We}{\om\Rey}\|\Lm^{-1}\sig^i_q\|_{L^2}^2
\right]^{\fr12}, \quad \mathrm{for}\quad q_1<q\le q_0,
\eeqno
where $q_0$ and $q_1$  are to be determined later.
\bigbreak
\noindent{\bf Estimates of ${Y}^i_q$, $q\in \Z$.}\par
\noindent{\em Step 1: low frequencies:$q\le q_1$.}\par
\noindent Taking the $L^2$ inner product of $\eqref{loc-piOB}_1$ and $\eqref{loc-piOB}_2$ with $u_q^i$ and $\sig^i_q$ respectively, using the divergence free condition $\dv u=0$, we obtain
\be\label{eq-54}
\frac12\fr{d}{dt}\|u^i_q\|_{L^2}^2+\fr{1-\om}{\Rey}\|\Lm u^i_q\|_{L^2}^2-\fr{1}{\Rey}(u^i_q|\sig_q^i)
=(\tl{f}_q^i|u^i_q),
\ee
and
\be\label{eq-55}
\fr12\fr{d}{dt}\|\sig^i_q\|_{L^2}^2+\fr{1}{\We}\|\sig^i_q\|^2_{L^2}-\fr{\om}{\We}(\Dl u^i_q|\sig^i_q)
=(\tl{h}_q^i|\sig^i_q).
\ee
To cancel the ``bad" cross term $(u^i_q|\sig_q^i)$ in \eqref{eq-54}, we need the following equality for the $L^2$ inner product $(u^i_q|\sig_q^i)$ of $u^i_q$ and $\sig_q^i$,
\be\label{eq-57}
\fr{d}{dt}( u^i_q|\sig^i_q)+\fr{\om}{\We}\|\Lm u^i_q\|_{L^2}^2-\fr{1}{\Rey}\|\sig^i_q\|_{L^2}^2
-\fr{1-\om}{\Rey}(\Dl u^i_q|\sig^i_q)+\fr{1}{\We}( u^i_q|\sig^i_q)=(\tl{f}_q^i|\sig^i_q)+(\tl{h}_q^i|u^i_q),
\ee
which is nothing but the sum of $\left(\eqref{loc-piOB}_1|\sig_q^i\right)$ and $\left(\eqref{loc-piOB}_2|u_q^i\right)$. Multiplying \eqref{eq-57} by $\fr{\We}{\Rey}$, adding the resulting equality to \eqref{eq-54}, one deduces that
\beq\label{eq57+}
\nn\frac12\fr{d}{dt}\left(\|u^i_q\|_{L^2}^2+2( u^i_q|\fr{\We}{\Rey}\sig^i_q)\right)&+&\fr{1}{\Rey}\|\Lm u^i_q\|_{L^2}^2-\fr{\We}{(\Rey)^2}\|\sig^i_q\|_{L^2}^2
\\&-&\fr{1-\om}{\Rey}(\Dl u^i_q|\fr{\We}{\Rey}\sig^i_q)
=(\tl{f}_q^i|u^i_q)+\fr{\We}{\Rey}\left((\tl{f}_q^i|\sig^i_q)+(\tl{h}_q^i|u^i_q)\right).
\eeq
Multiplying \eqref{eq-55} by $2\left(\fr{\We}{\Rey}\right)^2$, adding the resulting equality to \eqref{eq57+}, then the negative term $-\fr{\We}{(\Rey)^2}\|\sig^i_q\|_{L^2}^2$ in \eqref{eq57+} is absorbed, and we have
\beq\label{eq57++}
\nn&&\fr12\fr{d}{dt}\left[\| u^i_q\|_{L^2}^2+2\|\fr{\We}{\Rey}\sig^i_q\|_{L^2}^2+2( u^i_q|\fr{\We}{\Rey}\sig^i_q)\right]\\
\nn&&+\fr{1}{\Rey}\|\Lm u^i_q\|_{L^2}^2+\fr{1}{\We}\|\fr{\We}{\Rey}\sig^i_q\|_{L^2}^2-\fr{1+\om}{\Rey}(\Dl u^i_q|\fr{\We}{\Rey}\sig^i_q)\\
&=&(\tl{f}_q^i|u^i_q)+2\left(\fr{\We}{\Rey}\right)^2(\tl{h}_q^i|\sig^i_q)+\fr{\We}{\Rey}\left((\tl{f}_q^i|\sig^i_q)+(\tl{h}_q^i|u^i_q)\right),
\eeq
In order to exhibit the smoothing effect of $\sig^i, 1\le i\le d$ in the low frequency case, we shall make use of the linear coupling term $-\fr{1}{\Rey}\sig^i$ in the first equation of \eqref{piOB}. To do this, taking the $L^2$ inner product of $\eqref{loc-piOB}_1$ with $\Dl\sig^i_q$, and taking the $L^2$ inner product of $\eqref{loc-piOB}_2$ with $\Dl u^i_q$, then adding them together, we are led to
\beq\label{eq-58}
&&\nn\fr{d}{dt}(u^i_q|\Dl \sig^i_q)+\fr{1}{\Rey}\|\Lm\sig^i_q\|_{L^2}^2-\fr{\om}{\We}\|\Dl u^i_q\|_{L^2}^2-\fr{1-\om}{\Rey}(\Dl u^i_q|\Dl\sig^i_q)+\fr{1}{\We}(\Dl  u^i_q|\sig^i_q)\\
&=&(\tl{f}_q^i|\Dl\sig^i_q)+(\tl{h}_q^i|\Dl u^i_q)-(\dot{S}_{q-1}u\cdot\nb u^i_q|\Dl\sig^i_q)-(\dot{S}_{q-1}u\cdot\nb \sig^i_q|\Dl u^i_q).
\eeq
Multiplying  \eqref{eq-58} by $\fr{\We}{\Rey}$, adding the resulting equality to \eqref{eq57++} yields
\beq\label{eq67}
\nn&&\fr12\fr{d}{dt}\left[\| u^i_q\|_{L^2}^2+2\|\fr{\We}{\Rey}\sig^i_q\|_{L^2}^2+2( u^i_q|\fr{\We}{\Rey}\sig^i_q)+2( u^i_q|\fr{\We}{\Rey}\Dl\sig^i_q)\right]\\
\nn&&+\fr{1}{\Rey}\|\Lm u^i_q\|_{L^2}^2-\fr{\om}{\Rey}\|\Dl u^i_q\|_{L^2}^2+\fr{1}{\We}\|\fr{\We}{\Rey}\sig^i_q\|_{L^2}^2+\fr{1}{\We}\|\fr{\We}{\Rey}\Lm\sig^i_q\|^2_{L^2}\\
\nn&&-\fr{1-\om}{\Rey}(\Dl u^i_q|\fr{\We}{\Rey}\Dl\sig^i_q)+\left(\fr{1}{\We}-\fr{1+\om}{\Rey}\right)(\Dl u^i_q|\fr{\We}{\Rey}\sig^i_q)\\
\nn&=&(\tl{f}_q^i|u^i_q)+2\left(\fr{\We}{\Rey}\right)^2(\tl{h}_q^i|\sig^i_q)+\fr{\We}{\Rey}\left((\tl{f}_q^i|\sig^i_q)+(\tl{h}_q^i|u^i_q)\right)\\
&&+\fr{\We}{\Rey}\left((\tl{f}_q^i|\Dl\sig^i_q)+(\tl{h}_q^i|\Dl u^i_q)-(\dot{S}_{q-1}u\cdot\nb u^i_q|\Dl\sig^i_q)-(\dot{S}_{q-1}u\cdot\nb \sig^i_q|\Dl u^i_q)\right),
\eeq
Let us now estimate the cross terms contained in $Y_q^i$ for $q\le q_1$. Indeed, using Cauchy-Schwarz inequality and \eqref{equivalent}, one easily deduces that
\beno
2|(u^i_q|\fr{\We}{\Rey}\sig^i_q)|\le\fr58\|u^i_q\|^2_{L^2}+\fr85\|\fr{\We}{\Rey}\sig^i_q\|^2_{L^2},
\eeno
and
\beno
2|(u^i_q|\fr{\We}{\Rey}\Dl\sig^i_q)|\le\left(\fr832^{q_1}\right)^2\left(\|u^i_q\|^2_{L^2}+\|\fr{\We}{\Rey}\sig^i_q\|^2_{L^2}\right)\le\fr14\left(\|u^i_q\|^2_{L^2}+\|\fr{\We}{\Rey}\sig^i_q\|^2_{L^2}\right),
\eeno
provided
\be\label{eq-ll1}
2^{q_1}\le\fr{3}{16}.
\ee
Consequently,
\beqno
\| u^i_q\|_{L^2}^2+2\|\fr{\We}{\Rey}\sig^i_q\|_{L^2}^2+2( u^i_q|\fr{\We}{\Rey}\sig^i_q)+2( u^i_q|\fr{\We}{\Rey}\Dl\sig^i_q)
\ge\fr18\|u^i_q\|^2_{L^2}+\fr{3}{20}\|\fr{\We}{\Rey}\sig^i_q\|^2_{L^2},
\eeqno
and hence
\be\label{eq-ll2}
\left(Y^i_q\right)^2\approx\|u^i_q\|^2_{L^2}+\|\fr{\We}{\Rey}\sig^i_q\|^2_{L^2},\quad \mathrm{if}\quad q\le q_1,
\ee
with $q_1$ satisfying \eqref{eq-ll1}. Thanks to \eqref{equivalent}, the remaining two cross terms on the left hand side of \eqref{eq67} can be bounded as follows:
\beq\label{eq-20}
\left|\left(\fr{1}{\We}-\fr{1+\om}{\Rey}\right)(\Dl u^i_q|\fr{\We}{\Rey}\sig^i_q)\right|
\nn&\le&2\fr{\Rey+\We}{\We\Rey}\fr832^{q_1}\|\Lm u^i_q\|_{L^2}\left\|\fr{\We}{\Rey}\sig^i_q\right\|_{L^2}\\
\nn&\le&\fr832^{q_1}\fr{\Rey+\We}{\sqrt{\We\Rey}}\left(\fr{1}{\Rey}\|\Lm u^i_q\|^2_{L^2}+\fr{1}{\We}\left\|\fr{\We}{\Rey}\sig^i_q\right\|^2_{L^2}\right)\\
&\le&\fr14\left(\fr{1}{\Rey}\|\Lm u^i_q\|^2_{L^2}+\fr{1}{\We}\left\|\fr{\We}{\Rey}\sig^i_q\right\|^2_{L^2}\right),
\eeq
provided
\be\label{eq-ll3}
2^{q_1}\le\fr{3}{32}\fr{\sqrt{\We\Rey}}{\Rey+\We}.
\ee
Moreover,
\beq\label{eq-21}
\nn\fr{1-\om}{\Rey}\left|( \Dl u^i_q|\fr{\We}{\Rey}\Dl\sig^i_q)\right|
&\le&\fr{1}{2}\left(\fr83 2^{q_1}\right)^2\fr{1}{\Rey}\left(\|\Lm u^i_q\|_{L^2}^2+\left\|\fr{\We}{\Rey}\Lm\sig^i_q\right\|_{L^2}^2\right)\\
&\le&\fr18\fr{1}{\Rey}\|\Lm u^i_q\|_{L^2}^2+\fr12\fr{1}{\We}\left\|\fr{\We}{\Rey}\Lm\sig^i_q\right\|_{L^2}^2,
\eeq
provided
\be\label{eq-ll4}
2^{q_1}\le\fr{3}{16},\quad\textrm{and}\quad 2^{q_1}\le\fr{3}{8}\sqrt{\fr{\Rey}{\We}}.
\ee
Finally, using \eqref{equivalent} again yields
\be\label{eq-23}
\fr{\om}{\Rey}\|\Dl u^i_q\|_{L^2}^2\le\left(\fr83 2^{q_1}\right)^2\fr{1}{\Rey}\|\Lm u^i_q\|^2_{L^2}\le\fr14\fr{1}{\Rey}\|\Lm u^i_q\|^2_{L^2},
\ee
provided \eqref{eq-ll1} holds.
Now collecting \eqref{eq-ll1}, \eqref{eq-ll3} and \eqref{eq-ll4}, we take
\be\label{q_1}
q_1:=\left[\log_2\left(\fr{3}{32}\fr{\sqrt{\We\Rey}}{\Rey+\We}\right)\right].
\ee
From \eqref{eq-20}--\eqref{eq-23}, and using the left hand side of \eqref{equivalent}, we find that, for $q\le q_1$ with $q_1$ satisfying \eqref{q_1}, there holds
\beq\label{eq-24}
\nn&&\fr{1}{\Rey}\|\Lm u^i_q\|_{L^2}^2-\fr{\om}{\Rey}\|\Dl u^i_q\|_{L^2}^2+\fr{1}{\We}\|\fr{\We}{\Rey}\sig^i_q\|_{L^2}^2+\fr{1}{\We}\|\fr{\We}{\Rey}\Lm\sig^i_q\|^2_{L^2}\\
\nn&&-\fr{1-\om}{\Rey}(\Dl u^i_q|\fr{\We}{\Rey}\Dl\sig^i_q)+\left(\fr{1}{\We}-\fr{1+\om}{\Rey}\right)(\Dl u^i_q|\fr{\We}{\Rey}\sig^i_q)\\
\nn&\ge&\fr38\fr{1}{\Rey}\|\Lm u^i_q\|^2_{L^2}
+\fr12\fr{1}{\We}\left\|\fr{\We}{\Rey}\Lm\sig^i_q\right\|_{L^2}^2
+\fr34\fr{1}{\We}\left\|\fr{\We}{\Rey}\sig^i_q\right\|_{L^2}^2\\
&\ge&\fr38\left(\fr342^q\right)^2\min\left\{\fr{1}{\We},\fr{1}{\Rey}\right\}\left(\| u^i_q\|^2_{L^2}
+\left\|\fr{\We}{\Rey}\sig^i_q\right\|_{L^2}^2\right).
\eeq
To bound the right hand side of \eqref{eq67}, firstly integrating by parts and using the divergence free condition $\dv u=0$ yields
\be\label{eq-29}
-(\dot{S}_{q-1}u\cdot\nb u^i_q|\Dl\sig^i_q)-(\dot{S}_{q-1}\cdot\nb\sig^i_q|\Dl u^i_q)
 =(\dot{S}_{q-1}\pr_j u^k|\pr_ku^i_q\pr_j\sig^i_q+\pr_k\sig^i_q\pr_ju^i_q).
 \ee
Accordingly, using Bernstein's inequality, \eqref{q_1} and Lemma \ref{d-f}, we are led to
\beqno
&&\fr{\We}{\Rey}\left|(\dot{S}_{q-1}u\cdot\nb u^i_q|\Dl\sig^i_q)+(\dot{S}_{q-1}\cdot\nb\sig^i_q|\Dl u^i_q)\right|\\
&\le& C2^{2q_1}\|\nb u\|_{L^\infty}\|u^i_q\|_{L^2}\|\fr{\We}{\Rey}\sig^i_q\|_{L^2}
\le C\|u^h\|_{\dot{B}^{\fr{d}{2}+1}_{2,1}}^{\fr12}\|u\|_{\dot{B}^{\fr{d}{2}+1}_{2,1}}^{\fr12}\left(\|u^i_q\|_{L^2}^2+\|\fr{\We}{\Rey}\sig^i_q\|_{L^2}^2\right),
\eeqno
which, together with \eqref{eq-ll2}, implies that
\be\label{eq-ll-r}
\textrm{R.H.S. of\ } \eqref{eq67}\le C\left(\|\tl{f}^i_q\|_{L^2}+\|\fr{\We}{\Rey}\tl{h}_q^i\|_{L^2}+\|u^h\|_{\dot{B}^{\fr{d}{2}+1}_{2,1}}^{\fr12}\|u\|_{\dot{B}^{\fr{d}{2}+1}_{2,1}}^{\fr12}Y^i_q\right)Y^i_q.
 \ee
Substituting \eqref{eq-ll-r} into \eqref{eq67} ,  we infer from \eqref{eq-ll2} and \eqref{eq-24} that there exists a constant $C$ depending on $d, \Rey, \We$, but independent of $\om$, such that if $q\le q_1$, there holds
\be\label{eq68}
\fr{d}{dt}\left({Y}^i_q\right)^2+2^{2q}\left({Y}^i_q\right)^2\le CY^i_q\left(\|\tl{f}^i_q\|_{L^2}+\|\tl{h}_q^i\|_{L^2}+\|u^h\|_{\dot{B}^{\fr{d}{2}+1}_{2,1}}^{\fr12}\|u\|_{\dot{B}^{\fr{d}{2}+1}_{2,1}}^{\fr12}Y^i_q\right).
\ee

\noindent{\em Step 2: high frequencies:$q>q_1$.}\par
\noindent Part {\bf(i)}. $q>q_0.$\par
 First of all, we would like to point out that, in the high frequency case, the cross term $-\fr{\om}{\We}(\Dl u^i_q|\sig^i_q)$ in \eqref{eq-55} is  ``bad''. We will cancel it with the aid of \eqref{eq-57} once more. To this end, multiplying \eqref{eq-55} and \eqref{eq-57}  by $\left(\fr{(1-\om)\We}{\om\Rey}\right)^2$and $-\fr{(1-\om)\We}{\om\Rey}$ respectively, then adding the resulting equalities together yields
\beq\label{eq-60'}
\nn&&\fr12\fr{d}{dt}\left(\left\|\fr{(1-\om)\We}{\om\Rey}\sig^i_q\right\|_{L^2}^2-2\left(u^i_q\bigg|\fr{(1-\om)\We}{\om\Rey}\sig^i_q\right)\right)-\fr{1-\om}{\Rey}\|\Lm u^i_q\|_{L^2}^2\\
\nn&&+\fr{1}{(1-\om)\We}\left\|\fr{(1-\om)\We}{\om\Rey}\sig^i_q\right\|^2_{L^2}-\fr{1}{\We}\left(u^i_q\bigg|\fr{(1-\om)\We}{\om\Rey}\sig^i_q\right)\\
&=&\left(\fr{(1-\om)\We}{\om\Rey}\right)^2 (\tl{h}_q^i|\sig^i_q)-\fr{(1-\om)\We}{\om\Rey}\left((\tl{f}_q^i|\sig^i_q)+(\tl{h}_q^i|u^i_q)\right).
\eeq
To absorb the negative term  $-\fr{1-\om}{\Rey}\|\Lm u^i_q\|_{L^2}^2$ in \eqref{eq-60'}, we multiply \eqref{eq-54} by 2, and add the resulting equality to \eqref{eq-60'} to get
\beq\label{eq-60}
\nn&&\fr12\fr{d}{dt}\left(2\| u^i_q\|_{L^2}^2+\left\|\fr{(1-\om)\We}{\om\Rey}\sig^i_q\right\|_{L^2}^2-2\left(u^i_q\bigg|\fr{(1-\om)\We}{\om\Rey}\sig^i_q\right)\right)+\fr{1-\om}{\Rey}\|\Lm u^i_q\|_{L^2}^2\\
\nn&&+\fr{1}{(1-\om)\We}\left\|\fr{(1-\om)\We}{\om\Rey}\sig^i_q\right\|^2_{L^2}-\fr{2}{\Rey}(\sig^i_q|u^i_q)-\fr{1}{\We}\left(u^i_q\bigg|\fr{(1-\om)\We}{\om\Rey}\sig^i_q\right)\\
&=&2(\tl{f}_q^i|u^i_q)+\left(\fr{(1-\om)\We}{\om\Rey}\right)^2 (\tl{h}_q^i|\sig^i_q)-\fr{(1-\om)\We}{\om\Rey}\left((\tl{f}_q^i|\sig^i_q)+(\tl{h}_q^i|u^i_q)\right).
\eeq
It is easy to verify that
\be\label{eq-61}
2\| u^i_q\|_{L^2}^2+\left\|\fr{(1-\om)\We}{\om\Rey}\sig^i_q\right\|_{L^2}^2-2\left( u^i_q\bigg|\fr{(1-\om)\We}{\om\Rey}\sig^i_q\right)
\approx\| u^i_q\|_{L^2}^2+\left\|\fr{(1-\om)\We}{\om\Rey}\sig^i_q\right\|_{L^2}^2.
\ee
Using Cauchy-Schwarz inequality and \eqref{equivalent} yields
\beq\label{eq-62}
\nn&&\left|\fr{2}{\Rey}(\sig^i_q| u^i_q)+\fr{1}{\We}\left( u^i_q\bigg|\fr{(1-\om)\We}{\om\Rey}\sig^i_q\right)\right|\\
\nn&=&\fr{1+\om}{1-\om}\fr{1}{\We}\left|\left( u^i_q\bigg|\fr{(1-\om)\We}{\om\Rey}\sig^i_q\right)\right|\\
\nn&\le&\fr12\fr{1}{1-\om}\fr{1}{\We}\left\|\fr{(1-\om)\We}{\om\Rey}\sig^i_q\right\|_{L^2}^2+\fr{2}{(1-\om)\We}\| u^i_q\|_{L^2}^2\\
\nn&\le&\fr12\fr{1}{1-\om}\fr{1}{\We}\left\|\fr{(1-\om)\We}{\om\Rey}\sig^i_q\right\|_{L^2}^2+\left(\fr432^{-q_0}\right)^2\fr{2}{(1-\om)\We}\|\Lm  u^i_q\|_{L^2}^2\\
&\le&\fr12\fr{1}{1-\om}\fr{1}{\We}\left\|\fr{(1-\om)\We}{\om\Rey}\sig^i_q\right\|_{L^2}^2+\fr12\fr{1-\om}{\Rey}\|\Lm u^i_q\|_{L^2}^2,
\eeq
provided
\be\label{eq-63}
\left(\fr432^{-q_0}\right)^2\fr{2}{(1-\om)\We}\le\fr12\fr{1-\om}{\Rey},\qquad\mathrm{i. e.}\qquad 2^{q_0}\ge \fr{8}{3(1-\om)}\sqrt{\fr{\Rey}{\We}}.
\ee
Now let us take
\be\label{q_0}
q_0:=\left[\log_2\left(\fr{8}{3(1-\om)}\sqrt{\fr{\Rey}{\We}}\right)\right]+1.
\ee
Then it follows from \eqref{eq-62},  \eqref{eq-63} and \eqref{equivalent} that
\beq\label{eq-64}
\nn&&\fr{1-\om}{\Rey}\|\Lm u^i_q\|_{L^2}^2+\fr{1}{(1-\om)\We}\left\|\fr{(1-\om)\We}{\om\Rey}\sig^i_q\right\|^2_{L^2}\\
\nn&&-\fr{2}{\Rey}(\sig^i_q| u^i_q)-\fr{1}{\We}\left( u^i_q\bigg|\fr{(1-\om)\We}{\om\Rey}\sig^i_q\right)\\
\nn&\ge&\fr12\fr{1-\om}{\Rey}\|\Lm u^i_q\|_{L^2}^2+\fr12\fr{1}{(1-\om)\We}\left\|\fr{(1-\om)\We}{\om\Rey}\sig^i_q\right\|^2_{L^2}\\
\nn&\ge&\fr12\fr{1-\om}{\Rey}\left(\fr342^{q_0}\right)^2\| u^i_q\|_{L^2}^2+\fr12\fr{1}{(1-\om)\We}\left\|\fr{(1-\om)\We}{\om\Rey}\sig^i_q\right\|^2_{L^2}\\
&\ge&\fr12\fr{1}{(1-\om)\We}\left(\| u^i_q\|_{L^2}^2+\left\|\fr{(1-\om)\We}{\om\Rey}\sig^i_q\right\|^2_{L^2}\right).
\eeq
Finally, the right hand side of \eqref{eq-60} can be bounded  as follows,
\be\label{eq-65}
 \textrm{R.H.S. of\ } \eqref{eq-60}\le C\left(\|\tl{f}^i_q\|_{L^2}+\left\|\fr{(1-\om)\We}{\om\Rey}\tl{h}_q^i\right\|_{L^2}\right)\left(\| u^i_q\|_{L^2}+\left\|\fr{(1-\om)\We}{\om\Rey}\sig^i_q\right\|_{L^2}\right).
 \ee
Substituting \eqref{eq-65} into \eqref{eq-60}, using \eqref{eq-61} and \eqref{eq-64}, we find that there exists a constant $C$ independent of  $d, \Rey, \We$, and $\om$, such that if $q> q_0$, there holds
\beq\label{eq66}
\fr{d}{dt}\left(Y^i_q\right)^2+\fr{1}{(1-\om)\We}\left(Y^i_q\right)^2\le CY^i_q\left(\|\tl{f}^i_q\|_{L^2}+\left\|\fr{(1-\om)\We}{\om\Rey}\tl{h}_q^i\right\|_{L^2}\right).
\eeq

\noindent Part {\bf(ii)}. $q_1<q\le q_0$.\par
From the definitions  of $q_1$ and $q_0$ (see \eqref{q_1} and \eqref{q_0}), it is obvious  that there is a gap between the low frequency part $q\le q_1$ and the high frequency part $q>q_0$ since $q_1<q_0$. To overcome this difficulty, it is necessary to resort to some new observation. Considering that $\|\Lm^{-1}\sig^i_q\|_{L^2}$ is equivalent to $\|\sig^i_q\|_{L^2}$ if $q_1< q\le q_0$, we can use $\|\Lm^{-1}\sig^i_q\|_{L^2}$ instead of $\|\sig^i_q\|_{L^2}$ to fill up the gap. To do so, applying the operator $\Lm^{-1}$ to $\eqref{loc-piOB}_2$, and then taking the $L^2$ inner product of the resulting equation with $\Lm^{-1}\sig^i_q$, we obtain
\beq\label{eq-56}
\nn\fr12\fr{d}{dt}\|\Lm^{-1}\sig^i_q\|_{L^2}^2&+&\fr{1}{\We}\|\Lm^{-1}\sig^i_q\|^2_{L^2}+\fr{\om}{\We}(u^i_q|\sig^i_q)\\
&=&(\Lm^{-1}\tl{h}_q^i|\Lm^{-1}\sig^i_q)-(\Lm^{-1}(\dot{S}_{q-1}u\cdot\nb\sig^i_q)|\Lm^{-1}\sig^i_q).
\eeq
Multiplying \eqref{eq-56} by $\fr{\We}{\om\Rey}$, then adding the resulting equation to \eqref{eq-54}, the cross term $(u^i_q|\sig^i_q)$ is canceled, and we arrive at
\beq\label{eq69}
\nn&&\fr12\fr{d}{dt}\left(\| u^i_q\|_{L^2}^2+\fr{\We}{\om\Rey}\|\Lm^{-1}\sig^i_q\|_{L^2}^2
\right)+\fr{1-\om}{\Rey}\|\Lm u^i_q\|_{L^2}^2+\fr{1}{\om\Rey}\|\Lm^{-1}\sig^i_q\|_{L^2}^2\\
&\le&(\tl{f}_q^i|u^i_q)+\fr{\We}{\om\Rey}\left((\Lm^{-1}\tl{h}_q^i|\Lm^{-1}\sig^i_q)-(\Lm^{-1}(\dot{S}_{q-1}u\cdot\nb\sig^i_q)|\Lm^{-1}\sig^i_q)\right),
\eeq
In view of \eqref{equivalent} and \eqref{q_1}, it is not difficult to verify that
\beq\label{eq-n1}
\nn\fr{1-\om}{\Rey}\|\Lm u^i_q\|_{L^2}^2+\fr{1}{\om\Rey}\|\Lm^{-1}\sig^i_q\|_{L^2}^2&\ge& \fr{1-\om}{\Rey}\left(\fr342^{q_1}\right)^2\| u^i_q\|_{L^2}^2+\fr{1}{\om\Rey}\|\Lm^{-1}\sig^i_q\|_{L^2}^2\\
\nn&\ge&\fr{3^4}{2^{16}}\fr{\We(1-\om)}{(\Rey+\We)^2}\| u^i_q\|_{L^2}^2+\fr{1}{\om\Rey}\|\Lm^{-1}\sig^i_q\|_{L^2}^2\\
&\ge&\fr{3^4\We(1-\om)}{2^{16}(\Rey+\We)^2}\left(\| u^i_q\|_{L^2}^2+\fr{\We}{\om\Rey}\|\Lm^{-1}\sig^i_q\|_{L^2}^2\right).
\eeq
Using the divergence free condition $\dv u=0$, one easily deduces that
\beno
(\Lm^{-1}(\dot{S}_{q-1}u\cdot\nb\sig^i_q)|\Lm^{-1}\sig^i_q)\\
=([\Lm^{-1},\dot{S}_{q-1}u\cdot\nb]\sig^i_q)|\Lm^{-1}\sig^i_q),
\eeno
then Lemma \ref{lem-commu} and Lemma \ref{d-f} imply
\be\label{eq-51}
|(\Lm^{-1}(\dot{S}_{q-1}u\cdot\nb\sig^i_q)|\Lm^{-1}\sig^i_q)|\le C\|\nb \dot{S}_{q-1}u\|_{L^\infty}\|\Lm^{-1}\sig^i_q\|_{L^2}^2\le C\|u^h\|_{\dot{B}^{\fr{d}{2}+1}_{2,1}}^{\fr12}\|u\|_{\dot{B}^{\fr{d}{2}+1}_{2,1}}^{\fr12}\|\Lm^{-1}\sig^i_q\|_{L^2}^2.
\ee
Substituting \eqref{eq-n1} and \eqref{eq-51} into \eqref{eq69}, we find that there exists a constant
 $C$ depending on $d$,  but independent of $\Rey, \We, $ and $\om$, such that  if $q_1<q\le q_0$, there holds
\be\label{eq70}
\fr{d}{dt}\left({Y}^i_q\right)^2+\fr{3^4\We(1-\om)}{2^{16}(\Rey+\We)^2}\left({Y}^i_q\right)^2\le C Y^i_q\left(\|\tl {f}^i_q\|_{L^2}+\sqrt{\fr{\We}{\om\Rey}}\|\Lm^{-1}\tl {h}^i_q\|_{L^2}+\|u^h\|_{\dot{B}^{\fr{d}{2}+1}_{2,1}}^{\fr12}\|u\|_{\dot{B}^{\fr{d}{2}+1}_{2,1}}^{\fr12}Y_q^i\right).
\ee

\bigbreak
\noindent{\bf The smoothing effect of $u$.}\par
Now we are in a position to show the smoothing effect of the velocity $u$. In fact, in the low frequency part, the smoothing effect of $u_q$ has been obtained in \eqref{eq68}. More precisely, a direct consequence of \eqref{eq68} and \eqref{eq-ll2} gives
\beq\label{l}
\nn&&\left\| u^i_q(t)\right\|_{L^2}+\left\|\sig^i_q(t)\right\|_{L^2}+2^{2q}\left(\left\|u^i_q\right\|_{L^1_t(L^2)}+\left\|\sig^i_q\right\|_{L^1_t(L^2)}\right)\\
\nn&\le&C\Bigg(\left\| u^i_q(0)\right\|_{L^2}+\left\|\sig^i_q(0)\right\|_{L^2}+\left\|\tl{f}^i_q\right\|_{L^1_t(L^2)}+\left\|\tl{h}_q^i\right\|_{L^1_t(L^2)}\\
&&+\int_0^t\|u^h\|_{\dot{B}^{\fr{d}{2}+1}_{2,1}}^{\fr12}\|u\|_{\dot{B}^{\fr{d}{2}+1}_{2,1}}^{\fr12}\left(\| u^i_q\|_{L^2}+\left\|\sig^i_q\right\|_{L^2}\right)dt'\Bigg),
\eeq
for $q\le q_1$. Nevertheless, in the high frequency part,  we infer from \eqref{eq66}, \eqref{eq-61} and \eqref{eq70} that
\beq\label{hh}
\nn&&\| u^i_q(t)\|_{L^2}+\left\|\fr{(1-\om)\We}{\om\Rey}\sig^i_q(t)\right\|_{L^2}+\fr{1}{(1-\om)\We}\left(\| u^i_q\|_{L^1_t(L^2)}+\left\|\fr{(1-\om)\We}{\om\Rey}\sig^i_q\right\|_{L^1_t(L^2)}\right)\\
&\le&C\left(\| u^i_q(0)\|_{L^2}+\left\|\fr{(1-\om)\We}{\om\Rey}\sig^i_q(0)\right\|_{L^2}+\|\tl{f}^i_q\|_{L^1_t(L^2)}+\left\|\fr{(1-\om)\We}{\om\Rey}\tl{h}_q^i\right\|_{L^1_t(L^2)}\right),
\eeq
for $q>q_0$, and
\beq\label{m}
\nn&&\| u^i_q(t)\|_{L^2}+\sqrt{\fr{\We}{\om\Rey}}\left\|\Lm^{-1}\sig^i_q(t)\right\|_{L^2}+\fr{3^4\We(1-\om)}{2^{16}(\Rey+\We)^2}\left(\| u^i_q\|_{L^1_t(L^2)}+\sqrt{\fr{\We}{\om\Rey}}\left\|\Lm^{-1}\sig^i_q\right\|_{L^1_t(L^2)}\right)\\
\nn&\le&C\left(\| u^i_q(0)\|_{L^2}+\sqrt{\fr{\We}{\om\Rey}}\left\|\Lm^{-1}\sig^i_q(0)\right\|_{L^2}+\|\tl {f}^i_q\|_{L^1_t(L^2)}+\sqrt{\fr{\We}{\om\Rey}}\|\Lm^{-1}\tl {h}^i_q\|_{L^1_t(L^2)}\right.\\
&&\left.+\int_0^t\|u^h\|_{\dot{B}^{\fr{d}{2}+1}_{2,1}}^{\fr12}\|u\|_{\dot{B}^{\fr{d}{2}+1}_{2,1}}^{\fr12}\left(\| u^i_q\|_{L^2}+\sqrt{\fr{\We}{\om\Rey}}\left\|\Lm^{-1}\sig^i_q\right\|_{L^2}\right)dt'\right),
\eeq
for $q_1<q\le q_0$. Recalling the definitions of $q_1$ and $q_0$, and using \eqref{equivalent} again, one deduces that
\be\label{left}
\|\sig^i_q\|_{L^2}\le\fr832^{q_0}\|\Lm^{-1}\sig^i_q\|_{L^2}\le\fr{2^7}{3^2(1-\om)}\sqrt{\fr{\Rey}{\We}}\|\Lm^{-1}\sig^i_q\|_{L^2},
\ee
and
\be\label{right}
\|\sig^i_q\|_{L^2}\ge\fr342^{q_1}\|\Lm^{-1}\sig^i_q\|_{L^2}\ge\fr{3^2}{2^8}\fr{\sqrt{\Rey\We}}{\Rey+\We}\|\Lm^{-1}\sig^i_q\|_{L^2}.
\ee
By virtue of \eqref{left}, \eqref{m} implies
\beq\label{lh}
\nn&&\| u^i_q(t)\|_{L^2}+\fr{3^2(1-\om)\We}{2^7\sqrt{\om}\Rey}\left\|\sig^i_q(t)\right\|_{L^2}\\
\nn&&+\fr{3^4\We(1-\om)}{2^{16}(\Rey+\We)^2}\| u^i_q\|_{L^1_t(L^2)}+\fr{3^6(\We)^2(1-\om)^2}{\sqrt{\om}2^{23}(\Rey+\We)^2}\fr{1}{\Rey}\left\|\sig^i_q\right\|_{L^1_t(L^2)}\\
\nn&\le&C\left(\| u^i_q(0)\|_{L^2}+\sqrt{\fr{\We}{\om\Rey}}\left\|\Lm^{-1}\sig^i_q(0)\right\|_{L^2}+\|\tl {f}^i_q\|_{L^1_t(L^2)}+\sqrt{\fr{\We}{\om\Rey}}\|\Lm^{-1}\tl {h}^i_q\|_{L^1_t(L^2)}\right.\\
&&\left.+\int_0^t\|u^h\|_{\dot{B}^{\fr{d}{2}+1}_{2,1}}^{\fr12}\|u\|_{\dot{B}^{\fr{d}{2}+1}_{2,1}}^{\fr12}\left(\| u^i_q\|_{L^2}+\sqrt{\fr{\We}{\om\Rey}}\left\|\Lm^{-1}\sig^i_q\right\|_{L^2}\right)dt'\right),
\eeq
for $q_1< q\le q_0$. It can be seen from \eqref{hh} and \eqref{lh} that we have not yet obtain the soothing effect of $u_q$ if $q>q_1$. Fortunately,  the damping effect of $\sig_q$ is obtained both in \eqref{hh} and \eqref{lh}. This enables us to get the smoothing effect of $u_q$ by means of the equation of the velocity $u$.  To this end, taking the $L^2$ inner product of $\eqref{loc-piOB}_1$ with $u^i_q$, integrating by parts, we have
\beno
\fr12\fr{d}{dt}\|u^i_q\|_{L^2}^2+\fr{1-\om}{\Rey}\|\Lm u^i_q\|_{L^2}^2
=\fr{1}{\Rey}(\sig^i_q|u^i_q)+(\tl{f}^i_q|u^i_q).
\eeno
Using the left hand side of \eqref{equivalent}, we easily get
\be\label{su1}
\|u^i_q(t)\|_{L^2}+\fr{1-\om}{\Rey}\left(\fr342^{q}\right)^2\|u^i_q\|_{L^1_t(L^2)}
\le\|u^i_q(0)\|_{L^2}+\fr{1}{\Rey}\|\sig^i_q\|_{L^1_t(L^2)}+\|\tl{f}^i_q\|_{L^1_t(L^2)}.
\ee
Multiplying \eqref{hh} by $2\om$, and adding the resulting equation to \eqref{su1}, we arrive at
\beq\label{suhh}
\nn&&\| u^i_q(t)\|_{L^2}+\fr{2(1-\om)\We}{\Rey}\left\|\sig^i_q(t)\right\|_{L^2}+\fr{1-\om}{\Rey}\left(\fr342^{q}\right)^2\|u^i_q\|_{L^1_t(L^2)}+\fr{1}{\Rey}\left\|\sig^i_q\right\|_{L^1_t(L^2)}\\
&\le&C\left(\| u^i_q(0)\|_{L^2}+\left\|\sig^i_q(0)\right\|_{L^2}+\|\tl{f}^i_q\|_{L^1_t(L^2)}+\left\|\tl{h}_q^i\right\|_{L^1_t(L^2)}\right),
\eeq
for $q>q_0$.
Multiplying \eqref{lh} by $\fr{\sqrt{\om}2^{24}(\Rey+\We)^2}{3^6(\We)^2(1-\om)^2}$, and adding the resulting equation to \eqref{su1}, using \eqref{right}, we obtain
\beq\label{sulh}
\nn&&\| u^i_q(t)\|_{L^2}+\fr{2^{17}(\Rey+\We)^2}{3^4\Rey\We(1-\om)}\left\|\sig^i_q(t)\right\|_{L^2}+\fr{1-\om}{\Rey}\left(\fr342^{q}\right)^2\|u^i_q\|_{L^1_t(L^2)}+\fr{1}{\Rey}\left\|\sig^i_q\right\|_{L^1_t(L^2)}\\
\nn&\le&\fr{C}{(1-\om)^2}\Big(\| u^i_q(0)\|_{L^2}+\left\|\sig^i_q(0)\right\|_{L^2}+\|\tl {f}^i_q\|_{L^1_t(L^2)}\\
&&+\|\tl {h}^i_q\|_{L^1_t(L^2)}+\int_0^t\|u^h\|_{\dot{B}^{\fr{d}{2}+1}_{2,1}}^{\fr12}\|u\|_{\dot{B}^{\fr{d}{2}+1}_{2,1}}^{\fr12}\left(\| u^i_q\|_{L^2}+\left\|\sig^i_q\right\|_{L^2}\right)dt'\Big),
\eeq
for $q_1<q\le q_0$. Now collecting \eqref{suhh}, \eqref{sulh}, and \eqref{l}, we conclude that
\beq\label{u}
\nn&&\|u^i\|_{\tl{L}^\infty_t(\dot{B}^{\fr{d}{2}-1}_{2,1})}+(1-\om)\|\sig^i\|_{\tl{L}^\infty_t(\dot{B}^{\fr{d}{2}-1}_{2,1})}+(1-\om)\|u^i\|_{L^1_t(\dot{B}^{\fr{d}{2}+1}_{2,1})}+\|\sig^i\|_{L^1_t(\dot{B}^{\fr{d}{2}+1,\fr{d}{2}-1}_{2,1})}\\
\nn&\le&\fr{C}{(1-\om)^2}\left(\| u^i_0\|_{\dot{B}^{\fr{d}{2}-1}_{2,1}}+\left\|\sig^i_0\right\|_{\dot{B}^{\fr{d}{2}-1}_{2,1}}+\int_0^t\sum_{q\in\Z}2^{q(\fr{d}{2}-1)}\left(\|\tl {f}^i_q\|_{L^2}+\|\tl {h}^i_q\|_{L^2}\right)dt'\right.\\
&&\left.+\int_0^t\|u^h\|_{\dot{B}^{\fr{d}{2}+1}_{2,1}}^{\fr12}\|u\|_{\dot{B}^{\fr{d}{2}+1}_{2,1}}^{\fr12}\left(\| u^i\|_{\dot{B}^{\fr{d}{2}-1}_{2,1}}+\left\|\sig^i\right\|_{\dot{B}^{\fr{d}{2}-1}_{2,1}}\right)dt'\right),
\eeq
with some positive constant $C$ depending on $d, \Rey$ and $\We$, but independent of $\om$.

\noindent{\bf Estimates of $\sum_{q\in\Z}2^{q(\fr{d}{2}-1)}\left(\|\tl {f}^i_q\|_{L^2}+\|\tl {h}^i_q\|_{L^2}\right).$}\par
\noindent Recalling that $\tl{f}^i_q:=\dot{S}_{q-1}u\cdot\nb u^i_q+f^i_q$, and  $f^i:=-(\Pe(u\cdot\nb)u)^i=-(\dl_{im}+\pr_i\Lm^{-2}\pr_m)u^k\pr_ku^m$, in order to estimate  $\sum_{q\in\Z}2^{q(\fr{d}{2}-1)}\|\tl {f}^i_q\|_{L^2}$, it suffices to bound $\sum_{q\in\Z}2^{(\fr{d}{2}-1)}\|\dot{S}_{q-1}u^k\pr_ku^i_q\|_{L^2}$ and $\|(\dl_{im}+\pr_i\Lm^{-2}\pr_m)u^k\pr_ku^m\|_{\dot{B}^{\fr{d}{2}-1}_{2,1}}$.\par
\noindent{\em Estimates of $\sum_{q\in\Z}2^{(\fr{d}{2}-1)}\|\dot{S}_{q-1}u^k\pr_ku^i_q\|_{L^2}$.}
\begin{itemize}
\item Case 1: $1\le i\le d-1$.\par
\beqno
\sum_{q\in\Z}2^{(\fr{d}{2}-1)}\|\dot{S}_{q-1}u^k\pr_ku^i_q\|_{L^2}&\le&C\|\dot{S}_{q-1}u\|_{L^\infty}\sum_{q\in\Z}2^{\fr{d}{2}}\|u^h_q\|_{L^2}\\
&\le&C\|u\|_{L^\infty}\|u^h\|_{\dot{B}^{\fr{d}{2}}_{2,1}}\le C\left(\|u^h\|^2_{\dot{B}^{\fr{d}{2}}_{2,1}}+\|u^h\|_{\dot{B}^{\fr{d}{2}}_{2,1}}\|u^d\|_{\dot{B}^{\fr{d}{2}}_{2,1}}\right).
\eeqno

\item Case 2: $i=d$.\par
\beqno
\sum_{q\in\Z}2^{(\fr{d}{2}-1)}\|\dot{S}_{q-1}u^k\pr_ku^d_q\|_{L^2}&\le&C\sum_{q\in\Z}2^{(\fr{d}{2}-1)}\|\dot{S}_{q-1}u^h\pr_hu^d_q\|_{L^2}+\sum_{q\in\Z}2^{\fr{d}{2}-1}\|\dot{S}_{q-1}u^d\pr_du^d_q\|_{L^2}\\
&\le&C\|u^h\|_{L^\infty}\|u^d\|_{\dot{B}^{\fr{d}{2}}_{2,1}}+C\|u^d\|_{L^\infty}\|\dv_hu^h\|_{\dot{B}^{\fr{d}{2}-1}_{2,1}}\\
&\le&C\|u^h\|_{\dot{B}^{\fr{d}{2}}_{2,1}}\|u^d\|_{\dot{B}^{\fr{d}{2}}_{2,1}}.
\eeqno
\end{itemize}
Combining these two estimates, we are led to
\be\label{f1}
\sum_{q\in\Z}2^{(\fr{d}{2}-1)}\|\dot{S}_{q-1}u^k\pr_ku^i_q\|_{L^2}\le C\left(\|u^h\|^2_{\dot{B}^{\fr{d}{2}}_{2,1}}+\|u^h\|_{\dot{B}^{\fr{d}{2}}_{2,1}}\|u^d\|_{\dot{B}^{\fr{d}{2}}_{2,1}}\right)\quad\mathrm{for}\quad 1\le i\le d.
\ee
\noindent{\em Estimates of $\|(\dl_{im}+\pr_i\Lm^{-2}\pr_m)u^k\pr_ku^m\|_{\dot{B}^{\fr{d}{2}-1}_{2,1}}$.}\par\noindent
Using Proposition \ref{p-TR}, we have
\beqno
&&\|(\dl_{im}+\pr_i\Lm^{-2}\pr_m)\dot{T}_{u^k}\pr_ku^m\|_{\dot{B}^{\fr{d}{2}-1}_{2,1}}\\
&\le& \sum_{1\le m,k\le d-1}\|(\dl_{im}+\pr_i\Lm^{-2}\pr_m)\dot{T}_{u^k}\pr_ku^m\|_{\dot{B}^{\fr{d}{2}-1}_{2,1}}+ \sum_{1\le k\le d-1}\|(\dl_{id}+\pr_i\Lm^{-2}\pr_d)\dot{T}_{u^k}\pr_ku^d\|_{\dot{B}^{\fr{d}{2}-1}_{2,1}}\\
&&+\sum_{1\le m\le d-1}\|(\dl_{im}+\pr_i\Lm^{-2}\pr_m)\dot{T}_{u^d}\pr_du^m\|_{\dot{B}^{\fr{d}{2}-1}_{2,1}}+\|(\dl_{id}+\pr_i\Lm^{-2}\pr_d)\dot{T}_{u^d}\pr_du^d\|_{\dot{B}^{\fr{d}{2}-1}_{2,1}}\\
&\le& C\left(\|u^h\|_{L^\infty}\|\pr_hu^h\|_{\dot{B}^{\fr{d}{2}-1}_{2,1}}+\|u^h\|_{L^\infty}\|\pr_hu^d\|_{\dot{B}^{\fr{d}{2}-1}_{2,1}}+\|u^d\|_{L^\infty}\|\pr_du^h\|_{\dot{B}^{\fr{d}{2}-1}_{2,1}}+\|u^d\|_{L^\infty}\|\dv_hu^h\|_{\dot{B}^{\fr{d}{2}-1}_{2,1}}\right)\\
&\le& C\left(\|u^h\|^2_{\dot{B}^{\fr{d}{2}}_{2,1}}+\|u^h\|_{\dot{B}^{\fr{d}{2}}_{2,1}}\|u^d\|_{\dot{B}^{\fr{d}{2}}_{2,1}}\right),
\eeqno
and
\beqno
&&\|(\dl_{im}+\pr_i\Lm^{-2}\pr_m)\dot{T}'_{\pr_ku^m}u^k\|_{\dot{B}^{\fr{d}{2}-1}_{2,1}}\\
&\le& \sum_{1\le m,k\le d-1}\|(\dl_{im}+\pr_i\Lm^{-2}\pr_m)\dot{T}'_{\pr_ku^m}u^k\|_{\dot{B}^{\fr{d}{2}-1}_{2,1}}+ \sum_{1\le k\le d-1}\|(\dl_{id}+\pr_i\Lm^{-2}\pr_d)\dot{T}'_{\pr_ku^d}u^k\|_{\dot{B}^{\fr{d}{2}-1}_{2,1}}\\
&&+\sum_{1\le m\le d-1}\|(\dl_{im}+\pr_i\Lm^{-2}\pr_m)\dot{T}'_{\pr_du^m}u^d\|_{\dot{B}^{\fr{d}{2}-1}_{2,1}}+\|(\dl_{id}+\pr_i\Lm^{-2}\pr_d)\dot{T}'_{\pr_du^d}u^d\|_{\dot{B}^{\fr{d}{2}-1}_{2,1}}\\
&\le& C\left(\|\pr_hu^h\|_{\dot{B}^{\fr{d}{2}}_{2,1}}\|u^h\|_{\dot{B}^{\fr{d}{2}-1}_{2,1}}+\|\pr_hu^d\|_{\dot{B}^{\fr{d}{2}-1}_{2,1}}\|u^h\|_{\dot{B}^{\fr{d}{2}}_{2,1}}+\|\pr_du^h\|_{\dot{B}^{\fr{d}{2}-1}_{2,1}}\|u^d\|_{\dot{B}^{\fr{d}{2}}_{2,1}}+\|\dv_hu^h\|_{\dot{B}^{\fr{d}{2}-1}_{2,1}}\|u^d\|_{\dot{B}^{\fr{d}{2}}_{2,1}}\right)\\
&\le& C\left(\|u^h\|_{\dot{B}^{\fr{d}{2}-1}_{2,1}}\|u^h\|_{\dot{B}^{\fr{d}{2}+1}_{2,1}}+\|u^h\|_{\dot{B}^{\fr{d}{2}}_{2,1}}\|u^d\|_{\dot{B}^{\fr{d}{2}}_{2,1}}\right).
\eeqno
Therefore, according to Bony's decomposition, there holds
\be\label{f2}
\|(\dl_{im}+\pr_i\Lm^{-2}\pr_m)u^k\pr_ku^m\|_{\dot{B}^{\fr{d}{2}-1}_{2,1}}\le C\left(\|u^h\|_{\dot{B}^{\fr{d}{2}-1}_{2,1}}\|u^h\|_{\dot{B}^{\fr{d}{2}+1}_{2,1}}+\|u^h\|_{\dot{B}^{\fr{d}{2}}_{2,1}}\|u^d\|_{\dot{B}^{\fr{d}{2}}_{2,1}}\right)\quad\mathrm{for}\quad 1\le i\le d.
\ee
It follows from \eqref{f1}, \eqref{f2} and Proposition \ref{prop-classical} that
\beq\label{f}
\nn\sum_{q\in\Z}2^{q(\fr{d}{2}-1)}\|\tl {f}^i_q\|_{L^2}&\le& C\left(\|u^h\|_{\dot{B}^{\fr{d}{2}-1}_{2,1}}\|u^h\|_{\dot{B}^{\fr{d}{2}+1}_{2,1}}+\|u^h\|_{\dot{B}^{\fr{d}{2}}_{2,1}}\|u^d\|_{\dot{B}^{\fr{d}{2}}_{2,1}}\right)\\
&\le&C\left(\|u^h\|_{\dot{B}^{\fr{d}{2}-1}_{2,1}}\|u^h\|_{\dot{B}^{\fr{d}{2}+1}_{2,1}}+\|u^h\|^{\fr12}_{\dot{B}^{\fr{d}{2}+1}_{2,1}}\|u^d\|^{\fr12}_{\dot{B}^{\fr{d}{2}-1}_{2,1}}\|u^h\|^{\fr12}_{\dot{B}^{\fr{d}{2}-1}_{2,1}}\|u^d\|^{\fr12}_{\dot{B}^{\fr{d}{2}+1}_{2,1}}\right)\quad\mathrm{for}\quad 1\le i\le d.
\eeq
Next, we turn to estimate $\sum_{q\in\Z}2^{q(\fr{d}{2}-1)}\|\tl{h}^i_q\|_{L^2}$. Noticing that
\beqno
\tl{h}^i_q&=&\dot{S}_{q-1}u\cdot\nb \sig^i_q+h^i_q\\
&=&\dot{S}_{q-1}u\cdot\nb \sig^i_q-\ddl\left(\Pe\dv(u\cdot\nb)\tau+\Pe\dv g_\al(\tau,\nb u)\right)^i\\
&=&\left[\dot{S}_{q-1}u\cdot\nb \sig^i_q-\ddl\left(\Pe \dot{T}_{u^k}\pr_k\dv\tau\right)^i\right]-\ddl(\dl_{im}+\pr_i\Lm^{-2}\pr_m)\dot{T}_{\pr_lu^k}\pr_k\tau^{l,m}\\
&&-\ddl(\dl_{im}+\pr_i\Lm^{-2}\pr_m)\pr_l\dot{T}'_{\pr_k\tau^{l,m}}u^k-\ddl(\Pe\dv g_{\al}(\tau,\nb u))^i,
\eeqno
we bound $\sum_{q\in\Z}2^{q(\fr{d}{2}-1)}\|\tl{h}^i_q\|_{L^2}$ term by term. First of all, thanks to the commutate estimate Lemma 10.25 in  \cite{Bahouri-Chemin-Danchin11}, we find that
\beq\label{h1}
\nn&&\sum_{q\in\Z}2^{q(\fr{d}{2}-1)}\left\|\left[\dot{S}_{q-1}u\cdot\nb \sig^i_q-\ddl\left(\Pe \dot{T}_{u^k}\pr_k\dv\tau\right)^i\right]\right\|_{L^2}\\
&\le&C\|\nb u\|_{L^\infty}\sum_{q\in\Z}\sum_{|q'-q|\le4}2^{q(\fr{d}{2}-1)}\|\dot{\Dl}_{q'}\dv\tau\|_{L^2}\le C\|u^h\|^{\fr12}_{\dot{B}^{\fr{d}{2}+1}_{2,1}}\|u\|^{\fr12}_{\dot{B}^{\fr{d}{2}+1}_{2,1}}\|\tau\|_{\dot{B}^{\fr{d}{2}}_{2,1}}.
\eeq
The second term contained in $\sum_{q\in\Z}2^{q(\fr{d}{2}-1)}\|\tl{h}^i_q\|_{L^2}$ is nothing but $\|(\dl_{im}+\pr_i\Lm^{-2}\pr_m)\dot{T}_{\pr_lu^k}\pr_k\tau^{l,m}\|_{\dot{B}^{\fr{d}{2}-1}_{2,1}}$, the estimate of which is easier. Indeed, using Proposition \ref{prop-classical} and \ref{p-TR} directly yields
\beno
\|(\dl_{im}+\pr_i\Lm^{-2}\pr_m)\dot{T}_{\pr_lu^k}\pr_k\tau^{l,m}\|_{\dot{B}^{\fr{d}{2}-1}_{2,1}}\le C\|\nb u\|_{L^\infty}\|\nb\tau\|_{\dot{B}^{\fr{d}{2}-1}_{2,1}}\le C\|u^h\|^{\fr12}_{\dot{B}^{\fr{d}{2}+1}_{2,1}}\|u\|^{\fr12}_{\dot{B}^{\fr{d}{2}+1}_{2,1}}\|\tau\|_{\dot{B}^{\fr{d}{2}}_{2,1}}.
\eeno
Now we go to bound the third term,
\beno
\|(\dl_{im}+\pr_i\Lm^{-2}\pr_m)\pr_l\dot{T}'_{\pr_k\tau^{i,m}}u^k\|_{\dot{B}^{\fr{d}{2}-1}_{2,1}}\le C\|\dot{T}'_{\pr_k\tau}u^k\|_{\dot{B}^{\fr{d}{2}}_{2,1}}\le C\|\dot{T}'_{\pr_h\tau}u^h\|_{\dot{B}^{\fr{d}{2}}_{2,1}}+C\|\dot{T}'_{\pr_d\tau}u^d\|_{\dot{B}^{\fr{d}{2}}_{2,1}}.
\eeno
Using Proposition \ref{p-TR} once more and Bernstein's inequality, it is easy to see that
\be\label{h2}
\|\dot{T}'_{\pr_h\tau}u^h\|_{\dot{B}^{\fr{d}{2}}_{2,1}}\le C\|\pr_h\tau\|_{\dot{B}^{-1}_{\infty,\infty}}\|u^h\|_{\dot{B}^{\fr{d}{2}+1}_{2,1}}\le C\|u^h\|_{\dot{B}^{\fr{d}{2}+1}_{2,1}}\|\tau\|_{\dot{B}^{\fr{d}{2}}_{2,1}}.
\ee
However, the estimate of  $\|\dot{T}'_{\pr_d\tau}u^d\|_{\dot{B}^{\fr{d}{2}}_{2,1}}$ is a little bit complicated. In fact, using Lemma \ref{d-f} and Bernstein's inequality, we have
\beqno
\|\dot{T}'_{\pr_d\tau}u^d\|_{\dot{B}^{\fr{d}{2}}_{2,1}}&=&\sum_{q\in Z}2^{q{\fr{d}{2}}}\|\ddl \dot{T}'_{\pr_d\tau}u^d\|_{L^2}\\
&\le&\sum_{q\in Z}\sum_{q-q'\les1}2^{q{\fr{d}{2}}}\|\ddl (\dot{S}_{q'+2}{\pr_d\tau}\dot{\Dl}_{q'}u^d)\|_{L^2}\\
&\le&\sum_{q\in Z}\sum_{q-q'\les1}2^{q{\fr{d}{2}}}\|\dot{S}_{q'+2}{\pr_d\tau}\|_{L^\infty_{x_h}(L^2_{x_d})}\|\dot{\Dl}_{q'}u^d\|_{L^2_{x_h}(L^\infty_{x_d})}\\
&\le&C\sum_{q\in Z}\sum_{q-q'\les1}\sum_{q''\le q'+1}2^{q{\fr{d}{2}}}\|\dot{\Dl}_{q''}{\pr_d\tau}\|_{L^\infty_{x_h}(L^2_{x_d})}\|\dot{\Dl}_{q'}u^d\|^{\fr12}_{L^2}\|\dot{\Dl}_{q'}\dv_hu^h\|^{\fr12}_{L^2}\\
&\le&C\sum_{q\in Z}\sum_{q-q'\les1}\sum_{q''\le q'+1}2^{q{\fr{d}{2}}}2^{q''\fr{d+1}{2}}\|\dot{\Dl}_{q''}{\tau}\|_{L^2}\|\dot{\Dl}_{q'}u^d\|^{\fr12}_{L^2}\|\dot{\Dl}_{q'}\dv_hu^h\|^{\fr12}_{L^2}\\
&\le&C\sum_{q\in Z}\sum_{q-q'\les1}2^{(q-q')\fr{d}{2}}\sum_{q''\le q'+1}2^{(q''-q')\fr12}\left(2^{q''\fr{d}{2}}\|\dot{\Dl}_{q''}{\tau}\|_{L^2}\right)\\
&&\quad\quad\quad\quad\times\left(2^{q'(\fr{d}{2}+1)}\|\dot{\Dl}_{q'}u^d\|_{L^2}\right)^{\fr12}\left(2^{q'\fr{d}{2}}\|\dot{\Dl}_{q'}\dv_hu^h\|_{L^2}\right)^{\fr12}\\
&\le&C\|u^d\|^{\fr12}_{\dot{B}^{\fr{d}{2}+1}_{2,1}}\|u^h\|^{\fr12}_{\dot{B}^{\fr{d}{2}+1}_{2,1}}\|\tau\|_{\dot{B}^{\fr{d}{2}}_{2,1}}.
\eeqno
Finally, we would like to point out that the estimate of the  term $\|(\Pe\dv g_{\al}(\tau,\nb u))^i\|_{\dot{B}^{\fr{d}{2}-1}_{2,1}}$ is somehow ``bad'' due to the appearance of $\sum_{1\le k\le d-1}\|\dot{T}_{\tau^{l,k}}\pr_ku^d\|_{\dot{B}^{\fr{d}{2}}_{2,1}}$ both  for horizontal and vertical  direction. Thus, there is no need to bound $\|(\Pe\dv g_{\al}(\tau,\nb u))^i\|_{\dot{B}^{\fr{d}{2}-1}_{2,1}}$ from different directions. In view of Proposition \ref{prop-classical} and Corollary \ref{coro-product}, we easily get
\beq\label{h3}
\nn\|(\Pe\dv g_{\al}(\tau,\nb u))^i\|_{\dot{B}^{\fr{d}{2}-1}_{2,1}}&\le& C \|g_{\al}(\tau,\nb u)\|_{\dot{B}^{\fr{d}{2}}_{2,1}}\\
&\le& C\|\nb u\|_{\dot{B}^{\fr{d}{2}}_{2,1}}\|\tau\|_{\dot{B}^{\fr{d}{2}}_{2,1}}\le C\left(\| u^h\|_{\dot{B}^{\fr{d}{2}+1}_{2,1}}+\| u^d\|_{\dot{B}^{\fr{d}{2}+1}_{2,1}}\right)\|\tau\|_{\dot{B}^{\fr{d}{2}}_{2,1}}.
\eeq
Combining the above estimates, we find that
\be\label{h}
\sum_{q\in\Z}2^{q(\fr{d}{2}-1)}\|\tl{h}^i_q\|_{L^2}\le C\left(\| u^h\|_{\dot{B}^{\fr{d}{2}+1}_{2,1}}+\| u^d\|_{\dot{B}^{\fr{d}{2}+1}_{2,1}}\right)\|\tau\|_{\dot{B}^{\fr{d}{2}}_{2,1}}.
\ee
Substituting \eqref{f} and \eqref{h} into \eqref{u} yields
\beq\label{u'}
\nn&&\|u^i\|_{\dot{B}^{\fr{d}{2}-1}_{2,1}}+(1-\om)\|\sig^i\|_{\dot{B}^{\fr{d}{2}-1}_{2,1}}+(1-\om)\|u^i\|_{L^1_t(\dot{B}^{\fr{d}{2}+1}_{2,1})}+\|\sig^i\|_{L^1_t(\dot{B}^{\fr{d}{2}+1,\fr{d}{2}-1}_{2,1})}\\
\nn&\le&\fr{C}{(1-\om)^2}\left(\| u^i_0\|_{\dot{B}^{\fr{d}{2}-1}_{2,1}}+\left\|\sig^i_0\right\|_{\dot{B}^{\fr{d}{2}-1}_{2,1}}+\int_0^t\left(\| u^h\|_{\dot{B}^{\fr{d}{2}+1}_{2,1}}+\| u^d\|_{\dot{B}^{\fr{d}{2}+1}_{2,1}}\right)\|\tau\|_{\dot{B}^{\fr{d}{2}}_{2,1}}dt'\right.\\
\nn&&+\int_0^t\left(\|u^h\|_{\dot{B}^{\fr{d}{2}-1}_{2,1}}\|u^h\|_{\dot{B}^{\fr{d}{2}+1}_{2,1}}+\|u^h\|^{\fr12}_{\dot{B}^{\fr{d}{2}+1}_{2,1}}\|u^d\|^{\fr12}_{\dot{B}^{\fr{d}{2}-1}_{2,1}}\|u^h\|^{\fr12}_{\dot{B}^{\fr{d}{2}-1}_{2,1}}\|u^d\|^{\fr12}_{\dot{B}^{\fr{d}{2}+1}_{2,1}}\right)dt'\\
\nn&&\left.+\int_0^t\|u^h\|_{\dot{B}^{\fr{d}{2}+1}_{2,1}}^{\fr12}\|u\|_{\dot{B}^{\fr{d}{2}+1}_{2,1}}^{\fr12}\left(\| u^i\|_{\dot{B}^{\fr{d}{2}-1}_{2,1}}+\left\|\sig^i\right\|_{\dot{B}^{\fr{d}{2}-1}_{2,1}}\right)dt'\right)\\
\nn&\le&\fr{C}{(1-\om)^2}\left(\| u^i_0\|_{\dot{B}^{\fr{d}{2}-1}_{2,1}}+\left\|\sig^i_0\right\|_{\dot{B}^{\fr{d}{2}-1}_{2,1}}+\int_0^t\left(\| u^h\|_{\dot{B}^{\fr{d}{2}+1}_{2,1}}+\| u^d\|_{\dot{B}^{\fr{d}{2}+1}_{2,1}}\right)\|\tau\|_{\dot{B}^{\fr{d}{2}}_{2,1}}dt'\right.\\
\nn&&+\int_0^t\left(\|u^h\|_{\dot{B}^{\fr{d}{2}-1}_{2,1}}\|u^h\|_{\dot{B}^{\fr{d}{2}+1}_{2,1}}+\|u^h\|^{\fr12}_{\dot{B}^{\fr{d}{2}+1}_{2,1}}\|u^d\|^{\fr12}_{\dot{B}^{\fr{d}{2}-1}_{2,1}}\|u^h\|^{\fr12}_{\dot{B}^{\fr{d}{2}-1}_{2,1}}\|u^d\|^{\fr12}_{\dot{B}^{\fr{d}{2}+1}_{2,1}}\right)dt'\\
\nn&&\left.+\int_0^t\left(\|u^h\|_{\dot{B}^{\fr{d}{2}+1}_{2,1}}+\|u^h\|_{\dot{B}^{\fr{d}{2}+1}_{2,1}}^{\fr12}\|u^d\|_{\dot{B}^{\fr{d}{2}+1}_{2,1}}^{\fr12}\right)\| u^i\|_{\dot{B}^{\fr{d}{2}-1}_{2,1}}dt'\right)\quad \mathrm{for\ \ all }\quad i\in\{1,2,\cdots,d\} .
\eeq
Consequently, there exist two constants $C_1$ and $C_2$ depending on $d, \Rey$ and $\We$, but independent of $\om$, such that \eqref{uh} and \eqref{ud} hold.

\noindent{\bf The damping effect of $\tau$.}\par
\noindent In order to deal with the nonlinear coupling between $u$ and $\tau$, it is necessary to take advantage of the damping effect of $\tau$. To do so,  applying $\ddl$ to the equation of $\tau$ yields
\beqno
\pr_t \tau_q+\dot{S}_{q-1}u\cdot\nb \tau_q+\fr{\tau_q}{\We}=\fr{2\om}{\We}D(u_q)+\dot{S}_{q-1}u\cdot\nb \tau_q-\ddl(u\cdot\nb \tau)-\ddl g_\al(\tau,\nb u).
\eeqno
Taking the inner product of the above equation with $\tau_q$, integrating by parts, we have
\beqno
&&\fr12\fr{d}{dt}\|\tau_q\|_{L^2}^2+\fr{1}{\We}\|\tau_q\|_{L^2}^2\\
&=&\fr{2\om}{\We}(D(u_q)|\tau_q)+(\dot{S}_{q-1}u\cdot\nb \tau_q-\ddl(u\cdot\nb \tau)|\tau_q)-(\ddl g_\al(\tau,\nb u)|\tau_q).
\eeqno
It follows that
\beqno
\|\tau_q(t)\|_{L^2}+\fr{1}{\We}\|\tau_q\|_{L^1_t(L^2)}
&\le& C\left(\|\tau_q(0)\|_{L^2}+\om2^q\|u_q\|_{L^2}+\|\dot{S}_{q-1}u\cdot\nb \tau_q-\ddl(\dot{T}_{u^k}\pr_k \tau)\|_{L^2}\right.\\
&&\left.+\|\ddl \dot{T}'_{\pr_k\tau}u^k\|_{L^2}+\|\ddl g_\al(\tau,\nb u)\|_{L^2}\right).
\eeqno
Multiplying the above equation by $2^{q\fr{d}{2}}$, and summing over $q\in \Z$, we find that there exists a constant $C_3$ depending on $d, \Rey$ and $\We$, but independent of $\om$, such that
\beq\label{tau}
\nn&&\|\tau\|_{\tl{L}^\infty_t(\dot{B}^{\fr{d}{2}}_{2,1})}+\|\tau\|_{L^1_t(\dot{B}^{\fr{d}{2}}_{2,1})}\\
\nn&\le& C_3\left(\|\tau_0\|_{\dot{B}^{\fr{d}{2}}_{2,1}}+\om\|u\|_{L^1_t(\dot{B}^{\fr{d}{2}+1}_{2,1})}+\int_0^t\| u\|_{\dot{B}^{\fr{d}{2}+1}_{2,1}}\|\tau\|_{\dot{B}^\fr{d}{2}_{2,1}}dt'\right)\\
&\le&C_3\left(\|\tau_0\|_{\dot{B}^{\fr{d}{2}}_{2,1}}+\om\left(\|u^h\|_{L^1_t(\dot{B}^{\fr{d}{2}+1}_{2,1})}+\|u^d\|_{L^1_t(\dot{B}^{\fr{d}{2}+1}_{2,1})}\right)+\int_0^t(\|u^h\|_{\dot{B}^{\fr{d}{2}+1}_{2,1}}+\|u^d\|_{\dot{B}^{\fr{d}{2}+1}_{2,1}})\|\tau\|_{\dot{B}^\fr{d}{2}_{2,1}}dt'\right),
\eeq
where we have used the following nonlinear estimates which can be obtained in a similar manner to those of \eqref{h1}, \eqref{h2} and \eqref{h3},
\beno
\sum_{q\in\Z}2^{\fr{d}{2}q}\|\dot{S}_{q-1}u\cdot\nb \tau_q-\ddl(\dot{T}_{u^k}\pr_k \tau)\|_{L^2}\le C\|\nb u\|_{L^\infty}\|\tau\|_{\dot{B}^{\fr{d}{2}}_{2,1}},
\eeno
and
\beno
\| \dot{T}'_{\pr_k\tau}u^k\|_{\dot{B}^{\fr{d}{2}}_{2,1}}\le C\| u\|_{\dot{B}^{\fr{d}{2}+1}_{2,1}}\|\tau\|_{\dot{B}^{\fr{d}{2}}_{2,1}}, \qquad \|g_\al(\tau,\nb u)\|_{\dot{B}^{\fr{d}{2}}_{2,1}}\le C\|\nb u\|_{\dot{B}^{\fr{d}{2}}_{2,1}}\|\tau\|_{\dot{B}^{\fr{d}{2}-1}_{2,1}}.
\eeno
This completes the proof Proposition \ref{P1}.
\end{proof}
\begin{rem}\label{coro-h}
We can bound $\sum_{q\in\Z}2^{q(\fr{d}{2}-1)}\|\tl{h}^i_q\|_{L^2}$ without making a distinction between horizontal direction and vertical direction due to the appearance of the nonlinear   term $g_\al(\tau,\nb u)$. Our above estimates of $\sum_{q\in\Z}2^{q(\fr{d}{2}-1)}\|\tl{h}^i_q\|_{L^2}$ aim to reveal that the convection term $u\cdot\nb\tau$ behaves better than $g_\al(\tau,\nb u)$.
\end{rem}

\bigbreak

\noindent
The rest part of this section is to close the estimates \eqref{uh}--\eqref{tau'} in Proposition \ref{P1}. Before proceeding any further, let us denote
\beqno
A^d_0&:=&\fr{4C_2}{(1-\om)^3}\left(A^d(0)+1\right),\\
A^h_0&:=&\fr{16C_1C_2C_3}{(1-\om)^6}\exp\left(\fr{8(C^2_1+C_3)}{(1-\om)^6}\left(A^d_0\right)^2\right)\left(A^h(0)+\om A^d(0)+B(0)\right),\\ B_0&:=&\fr{192C_1C^2_2C^2_3}{(1-\om)^6}\exp\left(\fr{12(C_1^2+C_3)}{(1-\om)^6}\left(A^d_0\right)^2\right)\left( A^h(0)+\om A^d(0)+B(0)\right).
\eeqno

\begin{prop}\label{P2}
Assume that $(u,\tau)$ is a solution of system \eqref{IOBdimensionless} on $[0,T]$, satisfying
\be\label{Assumption1}
A^h(T)\le A^h_0,\quad A^d(T)\le A^d_0,\quad B(T)\le B_0.
\ee
If
\begin{gather}
\label{s1}\fr32A^h_0+\sqrt{A^h_0A^d_0}+B_0\le\fr{(1-\om)^3}{2C_2},\quad A^h_0\left(A^h_0+B_0\right)\le1,\\
\label{s2}B_0\left(1+\fr{2C_2}{(1-\om)^3}\left(A^h_0+B_0\right)\right)\le\fr{(1-\om)^3}{4C_1},\\
\label{s3}\fr32A^h_0\le1,\quad \fr{2C_2C_3}{(1-\om)^3}A^h_0\le\fr12,
\end{gather}
then there hold
\be\label{fr12Assumption1}
A^h(T)\le \fr12A^h_0,\quad A^d(T)\le \fr12A^d_0,\quad\mathrm{and}\quad B(T)\le \fr12B_0.
\ee
\end{prop}
\begin{proof}
From \eqref{ud}, it is easy to verify that
\beq\label{Ad1}
\nn A^d(t)&\le& \fr{C_2}{(1-\om)^3}\left(A^d(0)+A^h(t)\left(A^h(t)+B(t)\right)+\fr12A^h(t)\|u^d\|_{\tl{L}^\infty_t(\dot{B}^{\fr{d}{2}-1}_{2,1})}\right.\\
\nn&&\left.+\fr12A^h(t)\|u^d\|_{{L}^1_t(\dot{B}^{\fr{d}{2}+1}_{2,1})}+\left(A^h(t)+\sqrt{A^h(t)A^d(t)}\right)\|u^d\|_{\tl{L}^\infty_t(\dot{B}^{\fr{d}{2}-1}_{2,1})}+B(t)\|u^d\|_{{L}^1_t(\dot{B}^{\fr{d}{2}+1}_{2,1})}\right)\\
&\le&\fr{C_2}{(1-\om)^3}\left(A^d(0)+A^h(t)\left(A^h(t)+B(t)\right)+\left(\fr32A^h(t)+\sqrt{A^h(t)A^d(t)}+B(t)\right)A^d(t)\right).
\eeq
If
\be\label{H1}
\fr32A^h_0+\sqrt{A^h_0A^d_0}+B_0\le\fr{(1-\om)^3}{2C_2},
\ee
then it follows from \eqref{Assumption1} and \eqref{Ad1} that
\be\label{Ad2}
 A^d(t)\le\fr{2C_2}{(1-\om)^3}\left(A^d(0)+A^h(t)\left(A^h(t)+B(t)\right)\right),\quad t\in [0,T].
\ee
Furthermore, if
\be\label{H2}
A^h_0\left(A^h_0+B_0\right)\le1,
\ee
then \eqref{Ad2} changes to be
\be\label{Ad3}
 A^d(t)\le\fr{2C_2}{(1-\om)^3}\left(A^d(0)+1\right)=\fr12A^d_0,\quad t\in [0,T].
\ee
Next we go to bound $B(t)$. To do this, combining \eqref{tau} and \eqref{Ad2}, we are led to
\beq\label{B1}
B(t)\nn&\le& C_3\left(B(0)+\om A_0^h+\om A^d(t)+\int_0^t(\|u^h\|_{\dot{B}^{\fr{d}{2}+1}_{2,1}}+\|u^d\|_{\dot{B}^{\fr{d}{2}+1}_{2,1}})\|\tau\|_{\dot{B}^\fr{d}{2}_{2,1}}dt'\right)\\
\nn&\le& C_3\left(B(0)+ \om A_0^h+ \fr{2C_2\om}{(1-\om)^3}\left(A^d(0)+\left(A^h(t)\right)^2+A^h(t)B(t)\right)\right.\\
\nn&&\left.+\int_0^t(\|u^h\|_{\dot{B}^{\fr{d}{2}+1}_{2,1}}+\|u^d\|_{\dot{B}^{\fr{d}{2}+1}_{2,1}})\|\tau\|_{\dot{B}^\fr{d}{2}_{2,1}}dt'\right)\\
&\le& C_3\left(B(0)+ \fr32\om A_0^h+ \fr{2C_2\om}{(1-\om)^3}A^d(0)
+\int_0^t(\|u^h\|_{\dot{B}^{\fr{d}{2}+1}_{2,1}}+\|u^d\|_{\dot{B}^{\fr{d}{2}+1}_{2,1}})\|\tau\|_{\dot{B}^\fr{d}{2}_{2,1}}dt'\right)+\fr12B(t),
\eeq
provided
\be\label{H5}
\fr{2C_2C_3}{(1-\om)^3}A^h_0\le\fr12.
\ee
Using Gronwall's inequality, we obtain
\be\label{B2}
B(t)\le 2C_3\exp\left(2C_3(1+A^d_0)\right)\left(B(0)+ \fr32\om A_0^h+ \fr{2C_2\om}{(1-\om)^3}A^d(0)\right)\le\tl{B}_0,
\ee
where $\tl{B}_0$ is defined to be
\be\label{tlB}
\tl{B}_0:=2C_3\exp\left(2C_3(1+A^d_0)\right)\left(B(0)+ \fr32\om + \fr{2C_2\om}{(1-\om)^3}A^d(0)\right).
\ee
To estimate $A^h(t)$, we infer from \eqref{uh} that
\beq\label{Ah1}
\nn A^h(t)&\le&\fr{C_1}{(1-\om)^3}\left(A^h(0)+\int_0^t\left(\|u^h\|_{\dot{B}^{\fr{d}{2}+1}_{2,1}}+\fr{C_1}{(1-\om)^3}\|u^d\|_{\dot{B}^{\fr{d}{2}-1}_{2,1}}\|u^d\|_{\dot{B}^{\fr{d}{2}+1}_{2,1}}\right)\|u^h\|_{\dot{B}^{\fr{d}{2}-1}_{2,1}}dt'\right.\\
\nn&&+\fr{(1-\om)^3}{4C_1}\|u^h\|_{L^1_t(\dot{B}^{\fr{d}{2}+1}_{2,1})}+\fr12\int_0^t\left(\|u^h\|_{\dot{B}^{\fr{d}{2}+1}_{2,1}}+\|u^d\|_{\dot{B}^{\fr{d}{2}+1}_{2,1}}\right)\|u^h\|_{\dot{B}^{\fr{d}{2}-1}_{2,1}}dt'\\
\nn&&\left.+\|\tau\|_{L^\infty_t(\dot{B}^{\fr{d}{2}}_{2,1})}\| u^h\|_{L^1_t(\dot{B}^{\fr{d}{2}+1}_{2,1})}+\int_0^t\| u^d\|_{\dot{B}^{\fr{d}{2}+1}_{2,1}}\|\tau\|_{\dot{B}^{\fr{d}{2}}_{2,1}}dt'\right)\\
\nn&\le&\fr{C_1}{(1-\om)^3}\left(A^h(0)+\|\tau\|_{L^\infty_t(\dot{B}^{\fr{d}{2}}_{2,1})}\| u^d\|_{L^1_t(\dot{B}^{\fr{d}{2}+1}_{2,1})}+\left(\fr{(1-\om)^3}{4C_1}+B(t)\right)\|u^h\|_{L^1_t(\dot{B}^{\fr{d}{2}+1}_{2,1})}\right.\\
&&\left.+\int_0^t\left(\fr32\|u^h\|_{\dot{B}^{\fr{d}{2}+1}_{2,1}}+\left(\fr{C_1}{(1-\om)^3}A_0^d+\fr12\right)\|u^d\|_{\dot{B}^{\fr{d}{2}+1}_{2,1}}\right)\|u^h\|_{\dot{B}^{\fr{d}{2}-1}_{2,1}}dt'\right).
\eeq
In view of \eqref{Ad2}, we have
\beq\label{utau}
\nn\|\tau\|_{L^\infty_t(\dot{B}^{\fr{d}{2}}_{2,1})}\| u^d\|_{L^1_t(\dot{B}^{\fr{d}{2}}_{2,1})}&\le& \fr{2C_2}{(1-\om)^3}\left(A^d(0)+A^h(t)\left(A^h(t)+B(t)\right)\right)B(t)\\
&\le&\fr{2C_2}{(1-\om)^3}A^d(0)B(t)+\fr{2C_2}{(1-\om)^3}\left(A^h(t)+B(t)\right)B(t)A^h(t).
\eeq
Substituting \eqref{utau} into \eqref{Ah1}, and using \eqref{B2}, one deduces that
\beq\label{Ah2}
\nn A^h(t)\nn&\le&\fr{C_1}{(1-\om)^3}\left(A^h(0)+\fr{2C_2}{(1-\om)^3}A^d(0)\tl{B}_0+\left(B(t)\left(1+\fr{2C_2}{(1-\om)^3}\left(A^h(t)+B(t)\right)\right)\right)A^h(t)\right.\\
&&\left.+\int_0^t\left(\fr32\|u^h\|_{\dot{B}^{\fr{d}{2}+1}_{2,1}}+\left(\fr{C_1}{(1-\om)^3}A_0^d+\fr12\right)\|u^d\|_{\dot{B}^{\fr{d}{2}+1}_{2,1}}\right)
\|u^h\|_{\dot{B}^{\fr{d}{2}-1}_{2,1}}dt'\right)+\fr{1}{4}A^h(t).
\eeq
If
\be\label{H3}
B_0\left(1+\fr{2C_2}{(1-\om)^3}\left(A^h_0+B_0\right)\right)\le\fr{(1-\om)^3}{4C_1},
\ee
then \eqref{Ah2} reduces to
\beq\label{Ah3}
\nn A^h(t)&\le&\fr{2C_1}{(1-\om)^3}\left(A^h(0)+\fr{2C_2}{(1-\om)^3}A^d(0)\tl{B}_0\right.\\
&&\left.+\int_0^t\left(\fr32\|u^h\|_{\dot{B}^{\fr{d}{2}+1}_{2,1}}+\left(\fr{C_1}{(1-\om)^3}A_0^d+\fr12\right)\|u^d\|_{\dot{B}^{\fr{d}{2}+1}_{2,1}}\right)\|u^h\|_{\dot{B}^{\fr{d}{2}-1}_{2,1}}dt'\right).
\eeq
Gronwalll's inequality implies that
 \be\label{Ah4}
A^h(t)\le\fr{2C_1}{(1-\om)^3}\exp\left(\fr{2C_1}{(1-\om)^3}(1+A^d_0(1+\fr{C_1}{(1-\om)^3}A^d_0))\right)\left(A^h(0)+\fr{2C_2}{(1-\om)^3}A^d(0)\tl{B}_0\right),
\ee
provided
\be\label{H4}
\fr32A^h_0\le1.
\ee
Substituting \eqref{tlB} into \eqref{Ah4}, we find that
 \beq\label{Ah5}
\nn A^h(t)&\le&\fr{2C_1}{(1-\om)^3}\exp\left(\fr{2C_1^2}{(1-\om)^6}(1+A^d_0)^2\right)\Bigg\{A^h(0)\\
\nn&&+\fr{4C_2C_3}{(1-\om)^3}\exp\left(2C_3(1+A^d_0)\right)\left(A^d(0)B(0)+ \om A^d(0) A^d_0\right)\Bigg\}\\
\nn&\le&\fr{2C_1}{(1-\om)^3}\exp\left(\fr{8C_1^2}{(1-\om)^6}\left(A^d_0\right)^2\right)\Bigg\{A^h(0)+\fr{4C_2C_3}{(1-\om)^3}\exp\left(4C_3A^d_0\right)A^d_0\left(B(0)+ \om A^d(0) \right)\Bigg\}\\
&\le&\fr{8C_1C_2C_3}{(1-\om)^6}\exp\left(\fr{8(C_1^2+C_3)}{(1-\om)^6}\left(A^d_0\right)^2\right)\left(A^h(0)+\om A^d(0)+B(0)\right)=\fr12A^h_0.
\eeq
Finally, substituting \eqref{Ah5} into \eqref{B2}, we get the estimate of $B(t)$
in term of $A^h(0)+\om A^d(0)+B(0)$:
\beq\label{B}
\nn B(t)&\le& \fr{48C_1C_2C^2_3}{(1-\om)^6}\exp\left(4C_3A^d_0\right)\exp\left(\fr{8(C_1^2+C_3)}{(1-\om)^6}\left(A^d_0\right)^2\right)\left(2B(0)+ A^h(0)+2\om A^d(0)\right)\\
&\le& \fr{96C_1C^2_2C^2_3}{(1-\om)^6}\exp\left(\fr{12(C_1^2+C_3)}{(1-\om)^6}\left(A^d_0\right)^2\right)\left( A^h(0)+\om A^d(0)+B(0)\right)=\fr12B_0.
\eeq
The proof of Proposition \ref{P2} is completed.
\end{proof}

\section{Proof of Theorem\ref{thm-g}}
Now we are in a position to prove the maim result of this paper.
\begin{proof}
The local existence and uniqueness of the solution $(u,\tau)\in\mathcal{E}_T^{\fr{d}{2}}$ to system \eqref{IOBdimensionless} with initial data  $(u_0,\tau_0)\in \left(\dot{B}^{\fr{d}{2}-1}_{2,1}\right)^d\times \left(\dot{B}^{\fr{d}{2}}_{2,1}\right)^{d\times d}$ have been proved by Chemin and Masmoudi \cite{CM01}. Let $T^\star$ be the lifespan of $(u,\tau)$ obtained in \cite{CM01}. Define $T_1$ be the supremum of all time $T'\in[0,T^*)$ such that
\be\label{Assumption1'}
A^h(t)\le A^h_0,\quad A^d(t)\le A^d_0,\quad B(t)\le B_0\quad \mathrm{for\ \ all}\quad t\in[0,T'].
\ee
Choosing $C_0$ satisfying \eqref{data} so large that \eqref{s1}--\eqref{s3} hold, then thanks to Proposition \ref{P2}, for all $0\le t<T_1$, we have
\be\label{fr12Assumption1'}
A^h(t)\le \fr12A^h_0,\quad A^d(t)\le \fr12A^d_0,\quad B(t)\le \fr12B_0.
\ee
According to standard continuation method, we obtain $T_1=T^\star$. Moreover, \eqref{fr12Assumption1'} and the embedding $\dot{B}^{\fr{d}{2}}_{2,1}\hookrightarrow L^\infty$ imply that
\be\label{continuity}
\int_0^{T^\star}\left(\|\nb u(t,\cdot)\|_{L^\infty}+\|\tau(t,\cdot)\|_{L^\infty}\right)dt<\infty,
\ee
provided \eqref{data} holds. Thus, the blow-up criterion in \cite{CM01} enables us to conclude  that $T^\star=\infty$. This completes the proof of Theorem \ref{thm-g}.
\end{proof}

\bigbreak
\noindent{\bf Acknowledgments}
\bigbreak
Research supported by China Postdoctoral Science Foundation funded project 2014M552065, and National Natural Science Foundation of China 11401237, 11271322, and 11331005.

\end{document}